\documentclass[A4paper,reqno,oneside,12pt]{amsart}
\usepackage{fullpage}
\usepackage{setspace}
\usepackage{datetime}
\usepackage{tikz}
\usepackage{mathtools}
\usepackage{ mathrsfs }
\usetikzlibrary{er}
\usetikzlibrary{arrows.meta,decorations.pathmorphing,backgrounds,positioning,fit,petri}
\usepackage{amsmath,amsfonts,amssymb}
\usepackage{verbatim,enumitem,graphics}
\usepackage{microtype}
\usepackage{xcolor}
\usepackage[backref=page]{hyperref}
\hypersetup{%
    colorlinks,
    linkcolor={red!50!black},
    citecolor={green!50!black},
    urlcolor={blue!50!black}
}
\usepackage{dsfont}
\usepackage[abbrev,msc-links,backrefs]{amsrefs} 
\usepackage{doi}
\usepackage{lineno}
\usepackage[normalem]{ulem}

\newtheorem{theorem}{Theorem}[section]
\newtheorem{lemma}[theorem]{Lemma}
\newtheorem{sublemma}[theorem]{Sublemma}
\newtheorem{proposition}[theorem]{Proposition}
\newtheorem{fact}[theorem]{Fact}

\newtheorem{corollary}[theorem]{Corollary}

\newtheorem{definition}[theorem]{Definition}

\newtheorem{claim}[theorem]{Claim}

\newcommand{\boundellipse}[3]
{(#1) ellipse (#2 and #3)
}

\newcommand{\EE}{\mathbb E}

\newcommand{\PP}{\mathbb P}

\newcommand{\bj}{\mathbf{j}}

\newcommand{\cA}{\mathcal{A}}
\newcommand{\cB}{\mathcal{B}}
\newcommand{\cC}{\mathcal{C}}

\newcommand{\cJ}{\mathcal{J}}

\newcommand{\cL}{\mathcal{L}}

\newcommand{\cP}{\mathcal{P}}

\newcommand{\cT}{\mathcal{T}}

\newcommand{\cQ}{\mathcal{Q}}

\newcommand{\sP}{\mathscr{P}}

\newcommand{\oz}{\overline{z}}

\let\epsilon=\varepsilon

\begin{document}

\title{The $k$-representation number of the random graph}
\author[]{Ayush Basu}
\author[]{Vojt\v ech R\"odl}
\author[]{Marcelo Sales}
\begin{abstract}
   The $k$-representation number of a graph $G$ is the minimum cardinality of the system of vertex subsets with the property that every edge of $G$ is covered at least $k$ times while every non-edge is covered at most $(k-1)$ times. In particular, for $k=1$ this notion is equivalent to the clique number of a graph $G$. Extending results of \cite{https://doi.org/10.48550/arxiv.1103.4870} and \cite{EatonGrable} we study the $k$-representation number of $G(n,1/2)$. As a tool, we will prove a sharp concentration result counting the number of induced subgraphs of $G(n,1/2)$ with density $(\frac{1}{2}+\alpha)$. In Lemma \ref{thm: counting}, we will show that the number of such subgraphs is close to its expected value with probability $1-\exp(-n^C)$.
\end{abstract}

\address{Department of Mathematics, Emory University, Atlanta, GA, USA}
\email{\{ayush.basu|vrodl\}@emory.edu}
\address{Department of Mathematics, University of California, Irvine, CA, USA}
\email{mtsales@uci.edu}
\thanks{The first and second authors were supported by  by NSF grant DMS 2300347, the second and third authors were supported by NSF grant DMS 1764385, and the third author was also supported by US Air Force grant FA9550-23-1-0298.}

\maketitle

\section{Introduction}
\label{sec:Intro}
A \textit{set representation} of a graph $G$ consists of a set $S$ along with a family $(S_v)_{v\in V(G)}$, of subsets of $S$ indexed by the vertices of $G$,  such that
\begin{align*}
    |S_v\cap S_u|\geq 1 \text{ if and only if } uv \text{ is an edge of } G.
\end{align*} 
We say that $S$ \textit{represents} $G$ if there exists such a family of subsets of $S$ to form a set representation of $G$. The \textit{representation number} $\theta(G)$ is the smallest cardinality of a set $S$ that represents $G$. 

The representation number of $G$ depends on the structure of $G$ and has been of interest for a long time. A classical result due to Erd\H{o}s, Goodman and P\'osa \cite{EGP} from 1966 states that $\theta(G) \leq \lfloor n^2/4 \rfloor$ for any graph $G$. These authors also established a connection between the set representation of a graph and the clique covering of a graph. A collection $\cC = \{C_1,\dots,C_m\}$ of cliques that are subgraphs of $G$ is called a clique cover of $G$ if every edge is contained in one of the cliques in $\cC$. The clique covering number of $G$, denoted by $cc(G)$ is defined as minimum cardinality of a clique cover of $G$. It was shown in \cite{EGP} that $cc(G) = \theta(G)$. 

Since then the parameter $\theta(G)$ or equivalently $cc(G)$ has been studied for various graphs. Alon \cite{alon1986covering} gave an upper bound for $\theta(G)$ for graphs whose complements have bounded degree, and a lower bound for the same was given by Eaton and the second author \cite{eaton1996graphs}. Subsequently, Bollob\'as, Erd\H{o}s, Spencer and West \cite{bollobas1993clique} studied $\theta(G)$ for the random graph. The Erd\H{o}s--Renyi random graph on $n$ vertices is denoted by $G(n,p)$. It is a standard fact that in $G(n,1/2)$, the size of the largest clique is of the order of $\Theta(\log n)$, and hence each clique can cover at most $\Theta(\log^2 n)$ edges. Consequently, $\theta(G(n,1/2)) = \Omega(n^2/\log^2 n)$ with high probability. It was shown in  \cite{bollobas1993clique} that $\theta(G(n,1/2))= O(n^2 \log \log n/\log^2 n)$ with high probability. This was subsequently improved by Frieze and Reed \cite{https://doi.org/10.48550/arxiv.1103.4870}, Guo, Patton and Warnke \cite{guo2020prague}, see also \cite{Rdl2021}, to show that:
\begin{align*}
    \theta(G(n,1/2)) = \Theta\left(\frac{n^2}{\log^2 n}\right), 
\end{align*}
holds with high probability. The more general concept of $k$-representation where $k$ is a positive integer, has been studied by a number of authors \cite{EatonGrable, EatonGouldRodl, chung1994p,furedi97}. 

A \textit{$k$-representation} of a graph $G$ consists of a set $S$ with a family of subsets of $S$ indexed by the vertices of $G$, $(S_v)_{v\in V(G)}$ such that
\begin{align*}
    |S_v\cap S_u|\geq k \text{ if and only if } uv \text{ is an edge of } G
\end{align*} 
We say that $S$\textit{ represents} $G$ if there exists such a family of subsets of $S$ to form a set representation of $G$. The \textit{$k$ - representation number} $\theta_k(G)$ is the smallest cardinality of a set $S$ that represents $G$.  In particular, $\theta_1(G)$ is the representation number $\theta(G)$. 

Chung and West \cite{chung1994p}, and Eaton, Gould and the second author \cite{EatonGouldRodl} showed that $\theta_k(K_{n,n})=\Omega(n^2/k)$, and then F\"uredi \cite{furedi97}, showed that $\theta_k(K_{n,n})= O(n^2/k)$. Eaton and Grable \cite{EatonGrable} studied the order of magnitude of $\theta_k$ for the random graph $G(n,1/2)$. They showed that there exist absolute constants $A_1, A_2>0$ such that,   
\begin{align}
\label{eqn: theta_kprevbounds}
 \frac{A_1 n^2}{k^4\log ^2 n} \leq \theta_k(G(n,1/2))\leq \frac{A_2 n^2}{\log ^2 n},  
\end{align}
with high probability. While the upper bound follows from \cite{https://doi.org/10.48550/arxiv.1103.4870} by using the fact that $\theta_k(G)\leq \theta_1(G)+(k-1)$, they give a counting argument for the lower bound. 

We can see that as $k$ gets large, there is a considerable gap in the lower and upper bounds for $\theta_k(G(n,1/2))$ in \ref{eqn: theta_kprevbounds}. The main result of this paper is to shorten this gap for larger values of $k$. In fact, we provide an alternative proof for the lower bound for $\theta_k(G(n,1/2))$ that improves it by a factor of $k$ and improve the upper bound given for large values of $k$. 
\begin{theorem}
\label{thm:main}
There exist absolute positive constants $c_1, c_2 > 0$ such that for all $\epsilon > 0$, there exists $n_1=n_1(\epsilon)$, such that if $n\geq n_1$ and $1< k< n^{\frac{1}{2}-\epsilon}$,
\begin{align*}
\theta_k(G(n,1/2)) \geq \frac{c_1 n^2}{k^3\log^2 n} 
\end{align*}
with high probability and if $(\log \log n)^{1/\epsilon}\leq k\leq  \log n$,
\begin{align*}
\theta_k(G(n,1/2)) \leq \frac{c_2n^2}{k^{1-4\epsilon}\log^2 n}
\end{align*}
with high probability.
\end{theorem} 
The proof for the upper bound for Theorem \ref{thm:main} uses a lemma (Lemma \ref{thm: counting}) that counts the number of induced ``quasicliques" in $G(n,1/2)$ - subgraphs of size $c\alpha^{-2}\log n$ with density $(\frac{1}{2}+\alpha)$, where $c>0$ is an absolute constant. Roughly speaking, it shows that with exponentially high probability, the number of induced ``quasicliques" is within a factor of $(1\pm n^{-A\alpha^2})$ of the expected value, where $A>0$ is an absolute constant and $\alpha \geq (\log n)^{-1/2}$.
\\
The paper is organised as follows. In Section \ref{sec:sec2}, we restate the problem of $k$-representing a graph as a covering problem, define the notation used, and state the concentration inequalities used in the paper. In Section \ref{sec: sec2'}, we state the main Lemmas used to prove Theorem \ref{thm:main}. Sections \ref{sec: sec3} -- \ref{sec: sec7} are devoted to proving these Lemmas. 

\section{Equivalent Statement, Notation and Concentration Inequalities}
\label{sec:sec2}
\subsection{Equivalence with the condition of $k$-covering}
\label{sec: equivkcover}

Recall that the clique covering number $cc(G)$ is equal to the representation number of $\theta(G) = \theta_1(G)$. We will now define an analogous parameter for $\theta_k(G)$.

\begin{definition}[$k$-cover]
    \label{def: k-cover}
    Given a graph $G$ and a multiset $\cC = \{C_1,\dots, C_m\}$ of subsets of $V(G)$, we say that $\cC$ is a $k$-cover of $G$ if every edge of $G$ is contained in at least $k$-elements of $\cC$ and every non-edge is contained in at most $k-1$ elements of $\cC$. The $k$-covering number $\Phi_k(G)$ is the smallest cardinality of a $k$-cover.
\end{definition}
Notice that $\Phi_1(G) = cc(G)$. We now establish the following fact. 
\begin{fact}
    \label{thm: equivalence}
    For any graph $G$, $\theta_k(G) = \Phi_k(G)$. 
\end{fact}

\begin{proof}
    We prove this by establishing a correspondence between $k$-covers of $G$ and $k$-representations of a graph $G$.\
    For any $\cC=\{C_1,\dots, C_m\}$ which is a multiset of vertex subsets of $G$, and a vertex $u$ of $G$, let:
    \begin{equation*}
        S_u := \{i\in [m]: u\in C_i\}.
    \end{equation*}
    We say $(S_u)_{u\in V(G)}$ is a family of subsets of $[m]$ corresponding to $\cC$. Conversely, given a family $(S_u)_{u\in V(G)}$ of $[m]$, and $i\in [m]$, let 
    \begin{equation*}
        C_i := \{v\in V(G): i\in S_v\}.
    \end{equation*}
    We say $\cC = \{C_1,\dots, C_m\}$ is a multiset of subsets of $V(G)$ corresponding to $(S_u)_{u\in V(G)}$. 
    
    Note that  if $\cC = \{C_1,\dots, C_m\}$ corresponds to $(S_u)_{u\in V(G)}$, then $(S_u)_{u\in V(G)}$  corresponds to $\cC$, and further, for any subset of vertices $\{u,v\}\subseteq V(G)$ and any $i\in [m]$ we have that:
    \begin{equation*}
        \{u,v\}\subseteq C_i \iff i\in S_u\cap S_v.
    \end{equation*}
    Consequently, there is a bijection between $k$-covers of $G$ and $k$-representations and this implies $\theta_k(G) = \Phi_k(G)$.
\end{proof}
For the rest of the paper, we will consider the quantity $\Phi_k(G)$, and prove our statements about $k$-\textit{covers}. As an example, note that we have $\Phi_k(G)\leq \Phi_1(G)+ (k-1)$, because given any \textit{clique-cover}, one can add $k-1$ copies of $V(G)$ to it, to form a \textit{$k$-cover}. This gives the upper bound in (\ref{eqn: theta_kprevbounds}).

\subsection{Notation:} 
\label{sec: notation}
Let $G$ be a graph and $W\subseteq V(G)$ be a fixed subset of vertices. Given a multiset $\cC=\{C_1,\dots, C_m\}$ of subsets of $V(G)$, let $\cC(W)$ denote the multiset of $C_i$ that contain $W$. The cardinality $|\cC|$ of a multiset will always be the number of elements in it counted with multiplicity. As an example, if $\cC$ is a $k$-cover of $G$, for every edge $e$, $|\cC(e)|\geq k$ and for every non-edge $f$, $|\cC(f)|<k$. We will use $E(G)$ and $E(\overline{G})$ to denote edges and non-edges of $G$ respectively, and $e(W)$ to denote the number of edges induced by $W\subseteq V(G)$. We use the notation $x= a\pm b$ to denote $a-b\leq x \leq a+b$ where $0>b>a$. Finally, for the rest of the paper, we will use $c$ and $A$ to denote fixed absolute constants where, 
\begin{align*}
    c = \frac{1}{100\log 2}; \qquad A =10^{-4}c.
\end{align*}

\subsection{Concentration Inequalities}
\label{sec: inequalities}
Throughout the paper, at various times, we use the following standard concentration inequalities, which we state for easier reference. These are available in many textbooks (see for eg. Chapter 2 of Random Graphs by Janson, Luczak and Ruci\'nski \cite{janson2011random}). 
\begin{lemma}[Chernoff Bounds]
    \label{thm: Chernoff1}
    Let $X$ be a random variable with either binomial or hypergeometric distribution. Then for every $0<\epsilon < 1$,
    \begin{equation}
    \label{eqn: Chernoff1}
        \PP(|X-\EE[X]|\geq \epsilon\EE[X]) < 2\exp\left(-\frac{\epsilon^2\EE[X]}{3}\right).
    \end{equation}
    Further, if $\lambda> 7\EE[X]$, we have the stronger form:
    \begin{equation}
    \label{eqn: Chernoff2}
        \PP(X> \lambda) < \exp(-\lambda).
    \end{equation}
\end{lemma}

\begin{lemma}[Azuma Hoeffding Inequality]
    \label{thm: Azuma}
    Let $X_0, X_1, \dots, X_n$ be a martingale satisfying:
    \begin{equation*}
        |X_k - X_{k-1}|\leq c_k,
    \end{equation*}
    for all $1\leq k\leq n$. Then, we have that:
    \begin{equation}
        \label{eqn: Azuma}
        \PP(|X_n - X_0| > \lambda) < 2\exp\left(-\frac{\lambda^2}{2\sum_{i=1}^n c_i^2}\right).
    \end{equation}
\end{lemma}

\section{Proof of Theorem \ref{thm:main}}
\label{sec: sec2'}
The proof of both of our lower and upper bound follow these two broad steps. First, we define a property that $G(n,1/2)$ satisfies with high probability, and then we show our bounds for a graph $G$ that satisfies such a property. 

\subsection{Lower Bound:}
We use the following folklore property of $G(n,1/2)$ for the lower bound. 
\begin{definition}[Property $\sP_1$]
\label{def: pseudoLB}
    We say that $G$ on $n$ vertices satisfies property $\sP_1$ if for every subset of vertices $X$, the number of edges induced by $X$,
    \begin{align}
    \label{eqn: Fact3.1}
        e(X)= \left(\frac{1}{2}\pm 2\sqrt{\frac{\log n}{|X|}}\right){|X|\choose 2}.
    \end{align}
\end{definition}
\begin{fact}
    \label{thm: Fact3.1}
    $G(n,1/2)$ satisfies the Property $\sP_1$ with high probability.
\end{fact}
The proof of the lower bound in Theorem \ref{thm:main} will be a consequence of Fact \ref{thm: Fact3.1} and the following lemma, which will be proved in Section \ref{sec: sec3}. 
\begin{lemma}
\label{thm: Lemma3.2}
    For every $\epsilon> 0$, there exists $n_1 = n_1(\epsilon)$ such that for integers $n$ and $k$ satisfying $n\geq n_1$ and $1< k < n^{1/2-\epsilon}$, if $G$ on $n$ vertices satisfies Property $\sP_1$, then
    \begin{align*}
        \Phi_k(G) \geq \frac{c_1 n^2}{k^3\log ^2 n},
    \end{align*}
    for some absolute constant $c_1>0$. 
\end{lemma}

\subsection{Upper Bound:}
\label{sec: qcliqdef}
Now we will formulate the statements and definitions needed for the proof. Recall that, $c,A$ are absolute constants with values, 
\begin{align*}
    c = \frac{1}{100\log 2}; \qquad A =10^{-4}c.
\end{align*}
As noticed before, the $1$-cover of $G$ is the same as a clique cover. Thus for the case of $k=1$, the approach for the upper bound in \cite{Rdl2021, https://doi.org/10.48550/arxiv.1103.4870}, considered a covering of $G(n,1/2)$ by a collection of cliques of size $O(\log n)$ chosen by a nibble type method.

To construct a $k$-cover, (in view of Definition \ref{def: k-cover}), we need to find a collection $\cC$ of subsets of vertices of $V(G(n,1/2))$, such that each edge is contained in more members of $\cC$ than each non-edge. ``Typical" members of our choice for $\cC$ will be subsets of $t$ vertices with induced edge density $(\frac{1}{2}+\alpha)$ (where $\alpha,t$ are chosen depending on $k$). For technical reasons, we also need the graphs induced by the subsets of these vertices to be regular. We call such graphs $\alpha$-quasicliques and define them as follows. 
\begin{definition}
    \label{def: quasicliques}
    A graph on $t$ vertices is an $\alpha$-quasiclique if it has $\left(\frac{1}{2}+\alpha\right){t\choose 2}$ edges and every vertex has degree $\left(\frac{1}{2}+\alpha\pm \frac{3\alpha}{4}\right)t$.
\end{definition}

\begin{definition}
    \label{def: Q_G}
    Given a graph $G$, $\alpha > 0$ and positive integer $t$, $\cQ_G^{\alpha,t}$ is the collection of all subsets of $V(G)$ that induce $\alpha$-quasicliques of size $t$. 
\end{definition}

Whenever we use $\cQ_G^{\alpha,t}$, the choice of $\alpha$ and $t$ will be clear and so we will omit these parameters and denote the collection by $\cQ_G$. 

As stated above, as a first step for the proof of the upper bound, we define a suitable pseudorandom property for $G(n,1/2)$. 



\begin{definition}[$(\alpha,t)$-good graphs]
\label{def: alphatgood}
    Given $\alpha>0$ and integer $t$ and $T = (\frac{1}{2}+\alpha){t\choose 2}$, we say that $G$ on $n$ vertices is $(\alpha,t)$-good if for every $e\in E(G)$ and $f\in E(\overline{G})$, 
    \begin{align}
        \label{eqn: countinglemma1}
        |\cQ_G(e)| &= (1\pm n^{-A\alpha^2}){n-2\choose t-2}{{{t\choose 2}-1}\choose T-1}{2^{1-{t\choose 2}}}, \\
        \label{eqn: countinglemma2}
        |\cQ_G(f)| &= (1\pm n^{-A\alpha^2}){n-2\choose t-2}{{{t\choose 2}-1}\choose T}{2^{1-{t\choose 2}}}.
    \end{align}
    Further, 
    \begin{align}
    \label{eqn: countinglemma3}
         |\cQ_G| &= (1\pm n^{-A\alpha^2}){n\choose t}{{{t\choose 2}}\choose T}{2^{-{t\choose 2}}}.
    \end{align}
\end{definition}
Observe that, given $\alpha > 0$ and $t>0$, if $G$ is $(\alpha,t)$-good, then, in view of (\ref{eqn: countinglemma1}) and (\ref{eqn: countinglemma2}), for every edge $e$ and every non-edge $f$, $|\cQ_G(e)|/|\cQ(f)|$ is roughly equal to $(\frac{1}{2}+\alpha)/(\frac{1}{2}-\alpha)$ which is slightly larger than 1. Hence, every edge is covered by more $\alpha$-quasicliques of $G$ than every non-edge. The following lemma shows that $G(n,1/2)$ is $(\alpha,t)$-good with high probability. 
\begin{lemma}[Counting Lemma]
    \label{thm: counting} 
   For every $\epsilon > 0$, there exists $n_1:= n_1(\epsilon)$, such that whenever  $\alpha>0$ and integers $n,t>0$, satisfy 
   \begin{align}
    \label{eqn: 2.6condition0}
    n\geq n_1; \qquad \left(\frac{1}{\log n}\right)^{1/2-\epsilon} \leq \alpha \leq \frac{1}{2}; \qquad t = c\alpha^{-2}\log n,
    \end{align}
    $G(n,1/2)$ is $(\alpha,t)$-good with probability at least $1-2\exp(-n^{1/5})$. 
\end{lemma}
To prove the upper bound of Theorem \ref{thm:main}, for any $(\alpha,t)$-good $G$ (where $\alpha =\alpha(k)$ and $t=t(k)$ are chosen suitably), we construct a $k$-cover. As a first step, we choose a random subset $\cC$ of $\cQ_G^{\alpha,t}$ such that, 
\begin{align*}
    \EE[\cC(f)] < k < \EE[\cC(e)],
\end{align*}
whenever $e\in E(G)$ and $f\in E(\overline{G})$. Such a $\cC$ will not be a $k$-cover since some edges might be covered less than $k$ times and some non-edges might be covered more than $k-1$ times. To correct this, we will remove some members of $\cC$ and replace them by edges to obtain a $k$-cover. 
\begin{lemma}
\label{thm: upperbound}
    For every $\epsilon > 0$, there exists integer $n_1:= n_1(\epsilon)$ such that whenever $n,k$ are integers that satisfy $n\geq n_1$ and $(\log \log n)^{1/\epsilon}\leq k\leq \log n$, and $G$ is a $(k^{-\frac{1}{2}+\epsilon},ck^{1-2\epsilon}\log n)$-good graph on $n$ vertices, 
    \begin{align*}
        \Phi_k(G) \leq \frac{c_2 n^2}{k^{1-4\epsilon}\log ^2 n},
    \end{align*}
    for some absolute constant $c_2>0$. 
\end{lemma}
We prove Lemma \ref{thm: upperbound} in Section \ref{sec: sec4} and Lemma \ref{thm: counting} in Sections \ref{sec: sec5} -- \ref{sec: sec7}. 
\begin{proof}[Proof of Theorem \ref{thm:main}]
    Fix $\epsilon > 0$. For the lower bound, observe that Fact \ref{thm: Fact3.1} and Lemma \ref{thm: Lemma3.2} imply that if $n\geq n_1(\epsilon)$ and $1<k< n^{\frac{1}{2}-\epsilon}$,
    \begin{align*}
        \Phi_k(G(n,1/2)) \geq \frac{c_1 n^2}{k^3\log ^2 n},
    \end{align*}
    with high probability. 
    \\
    For the upper bound, if $n\geq n_1(\epsilon)$ and $(\log \log n)^{1/\epsilon}<k<\log n$, we let 
    \begin{align*}
        \alpha := k^{-\frac{1}{2}+\epsilon}; \qquad t := c\alpha^{-2}\log n = ck^{1-2\epsilon}\log n.
    \end{align*}
    Since by the above choice, 
    \begin{align*}
      \left(\frac{1}{\log n}\right)^{1/2-\epsilon}\leq \alpha \leq \frac{1}{\log\log n},
    \end{align*}
    Lemma \ref{thm: counting} implies that $G(n,1/2)$ is $(\alpha,t)=(k^{-\frac{1}{2}+\epsilon},ck^{1-2\epsilon}\log n)$-good with high probability. Together with Lemma \ref{thm: upperbound}, this implies: 
    \begin{align*}
        \Phi_k(G(n,1/2)) \leq \frac{c_2 n^2}{k^{1-4\epsilon}\log ^2 n}.
    \end{align*}
\end{proof}

\section{Proof of Lower Bound}
\label{sec: sec3}
We give a proof for of Fact \ref{thm: Fact3.1} for completeness. Recall that, we say that $G$ on $n$ vertices satisfies property $\sP_1$ if for every subset of vertices $X$, the number of edges induced by $X$,
    \begin{align}
    \label{eqn: Fact3.1}
        e(X)= \left(\frac{1}{2}\pm 2\sqrt{\frac{\log n}{|X|}}\right){|X|\choose 2}.
    \end{align}

\begin{proof}[Proof of Fact \ref{thm: Fact3.1}]
    Let $X\subseteq V = V(G(n,1/2))$. Then $e(X)$ is a binomial random variable with,
    \begin{align*}
        \EE[e(X)] = \frac{1}{2}{|X|\choose 2}.
    \end{align*}
    Let  $\epsilon = 4\sqrt{\frac{\log n}{|X|}}$. Observe that (\ref{eqn: Fact3.1}) holds trivially as long as $\epsilon \geq 1$, i.e $|X|\leq 16\log n$. Thus we assume, $|X|> 16\log n$ and thus $\epsilon < 1$. By (\ref{eqn: Chernoff1}), we have: 
    \begin{align*}
        \PP(|e(X)-\EE[e(X)]|>\epsilon\EE[e(X)]) &< 2\exp\left(-\frac{\epsilon^2\EE[e(X)]}{3}\right)\\
        & = 2\exp\left(-\frac{16\log n}{3|X|}\cdot \frac{|X|(|X|-1)}{4}\right)\\
        & = 2\exp\left(-\frac{4}{3}\log n \cdot (|X|-1)\right)< 2n^{-\frac{4}{3}(|X|-1)}.
    \end{align*}
    Thus for large enough $n$, the probability that there is a set of vertices $X$ with size greater than $16\log n$ and, 
    \begin{align*}
         \left|e(X)-\frac{1}{2}{|X|\choose 2}\right| >2\sqrt{\frac{\log n}{|X|}}{|X|\choose 2},
    \end{align*}
    is at most:
    \begin{align*}
        \displaystyle \sum_{\substack{X\subseteq V\\ |X|\geq 16\log n}} 2n^{-\frac{4}{3}(|X|-1)}& = \sum_{x = \lceil 16\log n \rceil}^n {n\choose x}\cdot 2n^{-\frac{4}{3}(x-1)}= o(1). 
    \end{align*}
    Thus $G(n,1/2)$ satisfies property $\sP_1$ with high probability. 
\end{proof} 
For the purpose of simplifying our calculations, we will use the following looser bounds when necessary. For any integers $x,n$ with $1\leq x\leq n$,
    \begin{align}
    \label{eqn: 3.1Approx}
          \left(\frac{1}{2}\pm 2\sqrt{\frac{\log n}{x}}\right){x\choose 2} \subseteq \left(\frac{1}{2}\pm 3\sqrt{\frac{\log n}{x}}\right)\frac{x^2}{2}.
    \end{align}
We will now prove Lemma \ref{thm: Lemma3.2}. Recall that we need to show that:\\
For every $\epsilon> 0$, there exists $n_1 = n_1(\epsilon)$ such that for integers $n$ and $k$ satisfying $n\geq n_1$ and $1< k < n^{1/2-\epsilon}$, if $G$ on $n$ vertices satisfies Property $\sP_1$, then
    \begin{align*}
        \Phi_k(G) \geq \frac{c_1 n^2}{k^3\log ^2 n},
    \end{align*}
for some absolute constant $c_1>0$. 

 For the rest of this section, fix $\epsilon > 0$ and let $k, n$ be positive integers satisfying $1< k < n^{\frac{1}{2}-\epsilon}$. Whenever necessary, we will assume that the integer $n_1 = n_1(\epsilon)$ is large enough and $n\geq n_1$. 

\subsection{Definition of Fractional Pseudocover.} In this section we will study a few properties of be a $k$-cover $\cC = \{C_1,\dots, C_m\}$ of a graph $G$ on $n\geq n_1$ vertices satisfying Property $\sP_1$. They will motivate the concept of a Fractional Pseudocover which we use to prove Lemma \ref{thm: Lemma3.2}. Let $x_i, e_i$ and $f_i$ be real numbers such that: 
\begin{align*}
    |C_i| = 9 x_i^2\log n; \quad e(C_i) = e_i; \quad {|C_i| \choose 2} - e_i = f_i. 
\end{align*}
By Property $\sP_1$ and (\ref{eqn: 3.1Approx}), we have the following: 
            \begin{align}
            \label{eqn: counte_i}
                e_i = \left(\frac{1}{2} \pm \frac{1}{x_i}\right)\frac{|C_i|^2}{2},\\
            \label{eqn: countf_i}
                f_i = \left(\frac{1}{2} \pm \frac{1}{x_i}\right)\frac{|C_i|^2}{2}.
            \end{align}
\begin{claim}
\label{thm: simpleLB}
If $x_i\leq k$ for all $i\in [m]$, then, 
            \begin{align*}
                m \geq \frac{1}{6}\frac{n^2}{k^3(9\log n)^2}.
            \end{align*}
\end{claim}
\begin{proof}
    
    By the fact that $G$ satisfies $\sP_1$ and since $\cC$ is a $k$-cover, we have $|\cC(e)|\geq k$ for every edge $e$. Thus we have:
  \begin{align}
    \label{eqn: sumofe_i1}
        \sum_{i=1}^m e_i &\geq k |E(G)|\geq k\cdot \left(\frac{1}{2} - 3\sqrt{\frac{\log n}{n}}\right)\frac{n^2}{2}.
    \end{align}
     Since $x_i\leq k$ for all  $i\in [m]$, then $e_i \leq  (9k^2\log n)^2$ and thus from (\ref{eqn: sumofe_i1}), we get: 
\begin{align*}
    m\geq \left(\frac{1}{2} -  3\sqrt{\frac{\log n}{n}}\right)\frac{kn^2}{2(9k^2\log n)^2} \geq \frac{1}{6}\frac{n^2}{k^3(9\log n)^2}.
\end{align*}
\end{proof}
We will see that in some sense Claim \ref{thm: simpleLB} describes the ``minimal cover". 

\begin{claim} 
\label{thm: FCPMotivation} 
For $n\geq n_1$, $k< n^{1/2-\epsilon}$, and $k$-covers of $G$ on $n$ vertices satisfying $\sP_1$, we have that $x_i$ satisfy the relations:
            \begin{align}
            \label{eqn: e_ivf_i}
                &\sum_{i=1}^m \left(x_i^3 - \frac{x_i^4}{8k-2}\right)\geq 0,\\
                 \label{eqn: sumofe_i}
                &\sum_{i=1}^m \left(x_i^3 + \frac{x_i^4}{2}\right)\geq \frac{k n^2}{3(9\log n)^2}.
            \end{align}
\end{claim}
\begin{proof}
    Similar to (\ref{eqn: sumofe_i1}), since $\cC$ is a $k$-cover of $G$, every non-edge $f$ satisfies $|\cC(f)|\leq (k-1)$. Together with Property $\sP_1$ this implies, 
    \begin{align}
        \label{eqn: sumoff_i1}
        \sum_{i=1}^m f_i &\leq (k-1)|E(\overline{G})| \leq  (k-1)\cdot\left(\frac{1}{2} +3\sqrt{\frac{\log n}{n}}\right)\frac{n^2}{2}.
    \end{align}
Note that for large enough $n$, since $k< n^{1/2-\epsilon}$,
\begin{align*}
   3\sqrt{\frac{\log n}{n}} < \frac{1}{8k-6}.
\end{align*}
From (\ref{eqn: sumofe_i1}) and (\ref{eqn: sumoff_i1}), we get that 
\begin{align}
\label{eqn: e_ivf_i1}
        \sum_{i=1}^m e_i \geq \frac{\left(\frac{1}{2}-\frac{1}{8k-6}\right)k}{\left(\frac{1}{2}+\frac{1}{8k-6}\right)(k-1)}\sum_{i=1}^m f_i = \frac{2k}{2k-1}\sum_{i=1}^m f_i.
\end{align}
Using (\ref{eqn: counte_i}) and (\ref{eqn: countf_i}), we infer that
\begin{align*}
     \sum_{i=1}^m \left(\frac{1}{2} + \frac{1}{x_i}\right)\frac{(9x_i^2\log n)^2}{2}\geq \sum_{i=1}^m e_i \geq \frac{2k}{2k-1}\sum_{i=1}^m f_i \geq \frac{2k}{2k-1} \sum_{i=1}^m \left(\frac{1}{2} - \frac{1}{x_i}\right)\frac{(9x_i^2\log n)^2}{2}.
\end{align*}
This implies 
\begin{align*}
    \sum_{i=1}^m\left(\frac{1}{x_i} \cdot \frac{4k-1}{2k-1} - \frac{1}{2}\cdot\frac{1}{2k-1}\right)\frac{(9\log n)^2 x_i^4}{2}\geq  0,
\end{align*}
and consequently,
\begin{align*}
    \sum_{i=1}^m\left(x_i^3 \cdot \frac{4k-1}{2k-1} - \frac{x_i^4}{2}\cdot\frac{1}{2k-1}\right) = \frac{4k-1}{2k-1}\sum_{i=1}^m\left(x_i^3  - \frac{x_i^4}{8k-2}\right)\geq 0,
\end{align*}
establishing (\ref{eqn: e_ivf_i}). Futher, from (\ref{eqn: counte_i}) and (\ref{eqn: sumofe_i1}), we have that:
    \begin{align*}
        \sum_{i=1}^m \left(\frac{1}{2} + \frac{1}{x_i}\right)\frac{(9x_i^2\log n)^2}{2}\geq \sum_{i=1}^m e_i \geq \frac{k}{3}\frac{n^2}{2},
    \end{align*}
    which implies (\ref{eqn: sumofe_i}). 
\end{proof}

To summarise, if all sets in the cover are ``small" (i.e., those with $x_i = O(k)$), then by Claim \ref{thm: simpleLB}, we will need ``many of them", i.e., as many as claimed in Lemma \ref{thm: Lemma3.2}. On the other hand, while the obvious advantage of large sets (those with $x_i = \Omega(k)$), is that they cover more edges, their disadvantage is that by Claim \ref{thm: FCPMotivation}, they contribute adversely to inequality (\ref{eqn: e_ivf_i}).

In the rest of the proof it will be convenient to work with fractional versions of properties (\ref{eqn: e_ivf_i}) and (\ref{eqn: sumofe_i}) of a $k$-cover. To do this, we introduce the more general concept of a \textit{Fractional Pseudocover}.

\begin{definition}[Fractional Pseudocover]
\label{def: FPC}
    For an integer $k>0$, a collection of pairs $\cA = \{A_1,\dots, A_m\}$ with $A_i = (x_i, w_i)$ is called a $k$-Fractional Pseudocover of $G$ on $n$ vertices if $0\leq w_i\leq 1$ and $x_i>0$ for all $i$ in $[m]$. Further, they must satisfy the following two properties:    
    \begin{enumerate}
        \item [(P1)]
        \begin{align}
            \sum_{i\in [m]}w_i \left(x_i^3 - \frac{x_i^4}{8k-2}\right)\geq 0,
        \end{align}
        \item [(P2)]
        \begin{align}
             \sum_{i\in [m]}w_i \left(x_i^3 + \frac{x_i^4}{2}\right)\geq \frac{k n^2}{3(9\log n)^2}.
        \end{align}
    \end{enumerate}
    We define $w(\cA):= w_1 + \dots + w_m$ to be the \textit{weight} of $\cA$.
\end{definition}

\subsection{Proof of Lemma \ref{thm: Lemma3.2}} To prove Lemma \ref{thm: Lemma3.2}, we use that every $k$-cover of $G$ is also a $k$-FPC of $G$ and then show a lower bound on the weight of all $k$-FPC of $G$. 
\begin{claim}
\label{thm: FPCKCover}
    Let $\cC = \{C_1,\dots, C_m\}$ be a $k$-cover of $G$ on $n\geq n_1$ vertices satisfying Property $\sP_1$. Then the collection $\cA = \{A_1,\dots, A_m\}$ where
    \begin{align*}
        A_i := \left( \sqrt{\frac{|C_i|}{9\log n}}, 1\right)
    \end{align*}
    is a $k$-\textit{FPC} of $G$ with $w(\cA) = m$.
\end{claim}
    \begin{proof}
        The proof follows immediately from Claim \ref{thm: FCPMotivation}. \textit{(P1)} follows from (\ref{eqn: e_ivf_i}), and \textit{(P2)} follows from (\ref{eqn: sumofe_i}). 
    \end{proof}

\begin{corollary} For every $G$ satisfying $\sP_1$ on $n\geq n_1$ vertices,  
    \begin{align*}
        \Phi_k(G)\geq \min \{w(\cA): \cA \text{ is a } k \text{- \textit{FPC} of } G \}
    \end{align*}
\end{corollary}
    The following proposition gives a lower bound on $w(\cA)$ for any $k$-\textit{FPC} $\cA$ of a graph $G$. Together with the above corollary, this implies Lemma \ref{thm: Lemma3.2}. 
    \begin{proposition}
        There exists an absolute constant $c_1>0$ such that if $G$ is a graph on $n\geq n_1$ vertices and $\cA$ is a $k$-FPC of $G$, then 
        \begin{align}
        \label{eqn: FPCLB}
            w(\cA) \geq \frac{c_1 n^2}{k^3\log ^2 n}.
        \end{align}
    \end{proposition}

    \begin{proof}
        Let $\cA= \{A_1,\dots, A_m\}$ with $A_i = (x_i, w_i)$ be a $k$-\textit{FPC} of $G$. Let $x_0 := 8k-2$ and let the set of ``large objects in $\cA$" be $\cL(\cA):=\{A_i: x_i > x_0, w_i>0\}$. 
        
        If $\cL(\cA)$ is empty, then $x_0 \geq x_i$ for every $i\in [m]$.  By \textit{(P2)}, this implies: 
        \begin{align*}
            (w_1 + \cdots + w_m)x_0^4 \geq  \sum_{i\in [m]}w_i \left(x_i^3 + \frac{x_i^4}{2}\right)\geq \frac{k n^2}{3(9\log n)^2}.
        \end{align*}
        Since $x_0 = 8k-2$, this implies (\ref{eqn: FPCLB}), i.e., 
        \begin{align}
        \label{eqn: finalstepLB}
           w(\cA) = \sum_{i=1}^m w_i \geq \frac{k n^2}{3x_0^4(9\log n)^2} \geq \frac{c_1 n^2}{k^3\log^2 n}, 
        \end{align}
        for sufficiently large $n$ and appropriate constant $c_1>0$.

        In the following two claims, we will show that if $\cL(\cA)$ is not empty, one can keep successively replacing $A_i$ in $\cL(\cA)$ to form $\cA'$ which is also a $k$-FPC, has the same weight, and $\cL(\cA')$ is empty. This will complete the proof of the Proposition. 
        \begin{claim}
        \label{thm: claim3.6}
            Let $A_m$ be in $\cL(\cA)$. Then there exists $S\subseteq [m-1]$ and $0< \beta_i\leq w_i$ for each $i$ in $S$, such that $x_i \leq x_0$ for all $i\in S$ and, 
            \begin{align}
            \label{eqn: balancing}
                \sum_{i\in S}\beta_i\left(x_i^3 - \frac{x_i^4}{x_0}\right) + w_m \left(x_m^3 - \frac{x_m^4}{x_0}\right) = 0.
            \end{align}
        \end{claim}
        \begin{proof}
             As defined before, $x_0 = 8k-2$. From \textit{(P1)}, we know that
             \begin{align}
                 \sum_{i\in [m]}w_i\left(x_i^3 - \frac{x_i^4}{x_0}\right)\geq 0.
             \end{align}
             and since $x_m > x_0$ we have:
             \begin{align*}
                \left(x_m^3 - \frac{x_m^4}{x_0}\right) <0.
             \end{align*}
             Therefore there must exist $S\subseteq [m-1]$ such that $x_i <x_0$  for every $i$ in $S$, and
             \begin{align}
                  \sum_{i\in S}w_i\left(x_i^3 - \frac{x_i^4}{x_0}\right) + w_m \left(x_m^3 - \frac{x_m^4}{x_0}\right) \geq 0.
             \end{align}
             Choosing an appropriate $\lambda$, such that $0< \lambda \leq 1$ and setting $\beta_i = \lambda w_i$ for all $i\in S$, one obtains (\ref{eqn: balancing}).
        \end{proof}
        \begin{claim}
        \label{thm: claim3.7}
            There is a $k$-FPC $\cA'$ with $w(\cA') = w(\cA)$ and $\cL(\cA')=\emptyset$. 
        \end{claim}
        \begin{proof}
            To prove the claim, it is enough to give a $k$-\textit{FPC} $\cA'$ such that $w(\cA') = w(\cA)$ and $|\cL(\cA')|<|\cL(\cA)|$. Therefore, one can keep iterating until $\cL(\cA')=\emptyset$.      
            
            Assume that $|\cL(\cA)|\geq 1$, otherwise $\cA'  = \cA$. Let us assume that $A_m \in \cL(\cA)$. Using Claim \ref{thm: claim3.6}, we obtain an $S\subseteq [m-1]$ and the collection $(\beta_i:i\in S)$ satisfying (\ref{eqn: balancing}). \\
            We define a new collection of pairs $\cA'$ as follows. Set 
            $$w_0' := w_m + \sum_{i\in S}\beta_i,$$ 
            and let $A_0' = (x_0, w_0')$. Further, for all $i\in [m-1]$ set,
            \begin{align*}
                w_i' = w_i -\beta_i \text{ for } i\in S,\text{ and } w_i' = w_i \text{ for }  i \in [m-1]\setminus S,
            \end{align*}
            and let $A_i' = (x_i, w_i')$ for all $i$ in $[m-1]$. Let $\cA' = \{A_0',\dots, A_{m-1}'\}$. Then we have that:
            \begin{align*}
                w(\cA') &= \left(w_m + \sum_{i\in S}\beta_i\right) + \sum_{i\in S} (w_i-\beta_i) + \sum_{i\in [m-1]\setminus S} w_i = w(\cA).
            \end{align*}
            Since $A_0'\notin \cL(\cA')$, $\cL(\cA')=\cL(\cA) \setminus\{A_m\}$. Consequently $|\cL(\cA')|= |\cL(\cA)|-1$.    
            Now it remains to show that $\cA'$ satisfies \textit{(P1)} and \textit{(P2)} in Definition \ref{def: FPC}. We have:
            \begin{align}
                \sum_{i = 0}^{m-1} w_i'\left(x_i^3 - \frac{x_i^4}{x_0}\right) - \sum_{i=1}^m  w_i\left(x_i^3 - \frac{x_i^4}{x_0}\right)= -\sum_{i\in S}\beta_i\left(x_i^3 - \frac{x_i^4}{x_0}\right) - w_m \left(x_m^3 - \frac{x_m^4}{x_0}\right).
            \end{align}
            Since, due to our choice of $S$ from Claim \ref{thm: claim3.6}, this quantity equals $0$, we infer that $\cA'$ also satisfies \textit{(P1)}. It remains to show that \textit{(P2)} holds for $\cA'$ too. By our choice of $\cA'$,
            \begin{align*}
                &\sum_{i=0}^{m-1}w_i' \left(x_i^3 + \frac{x_i^4}{2}\right)  -\sum_{i=1}^m w_i\left(x_i^3 + \frac{x_i^4}{2}\right)\\ 
                &= w_0'\left(x_0^3 + \frac{x_0^4}{2}\right) - \left(\sum_{i\in S}\beta_i \left(x_i^3 + \frac{x_i^4}{2}\right) + w_m \left(x_m^3 + \frac{x_m^4}{2}\right)\right),
            \end{align*} 
            and since $\cA$ satisfies \textit{(P2)}, proving the following inequality will complete the proof.
            \begin{align}
            \label{eqn: finalstep3.7}
                w_0'\left(x_0^3 + \frac{x_0^4}{2}\right) \geq \sum_{i\in S}\beta_i \left(x_i^3 + \frac{x_i^4}{2}\right) + w_m \left(x_m^3 + \frac{x_m^4}{2}\right).
            \end{align}
            From (\ref{eqn: balancing}), we have:
            \begin{align}
                \label{eqn: 3.21}
                \sum_{i\in S}\beta_i x_i^3 + w_m x_m^3 = \frac{\sum_{i\in S}\beta_i x_i^4 + w_m x_m^4}{x_0} = M.
            \end{align}
           By the Power Mean Inequality, we get: 
           \begin{align*}
               \left(\frac{\sum_{i\in S}\beta_i x_i^4 + w_m x_m^4}{w_0'}\right)^{1/4} &\geq \left(\frac{\sum_{i\in S}\beta_i x_i^3 + w_m x_m^3}{w_0'}\right)^{1/3}
           \end{align*}
           which can be rewritten as:
           \begin{align*}
               \left(\frac{Mx_0}{w_0'}\right)^{1/4}\geq \left(\frac{M}{w_0'}\right)^{1/3}.
           \end{align*}
          Consequently, 
           \begin{align}
            \left(w_0'x_0^3 \right)^{1/12} \geq M^{1/12}\nonumber,
            \end{align}
            which means
            \begin{align}
            \label{eqn: sumofcubes}
              w_0' x_0^3 \frac{}{}\geq \sum_{i\in S}\beta_i x_i^3 + w_m x_m^3 = M,
            \end{align}
            and hence also, 
            \begin{align}
            \label{eqn: sumoffourths}
            w_0' x_0^4 \frac{}{}\geq Mx_0 = \sum_{i\in S}\beta_i x_i^4 + w_m x_m^4.
           \end{align}
           Together, (\ref{eqn: sumofcubes}) and (\ref{eqn: sumoffourths}) imply (\ref{eqn: finalstep3.7}) and completes the proof. 
        \end{proof}\end{proof}

\section{Proof of Upper Bound}
\label{sec: sec4}
In this section, we will prove Lemma \ref{thm: upperbound}. For this section fix $\epsilon > 0$, and let $n,k, t$ be positive integers satisfying,
\begin{align}
\label{sec4values}
    n\geq n_1; \qquad (\log \log n)^{1/\epsilon}\leq k \leq \log n; 
\end{align}
and let, 
\begin{align}
\label{eqn: alphatink}
    \alpha = k^{-\frac{1}{2}+\epsilon}; \quad t = c\alpha^{-2}\log n.
\end{align} 
Whenever necessary, we will assume that the integer $n_1 = n_1(\epsilon)$ is large enough. 
\\
Let $T = (\frac{1}{2}+\alpha){t\choose 2}$ and,
    \begin{align}
        \label{eqn: n_0}
        N_0 &:= {n\choose t}{{t\choose 2}\choose T}2^{-{t\choose 2}},\\
        \label{eqn: n_1}
        N_1 &:= {n-2\choose t-2}{{t\choose 2}-1\choose T-1}2^{1-{t\choose 2}},\\
        \label{eqn: n_2}
        N_2 &:= {n-2\choose t-2}{{t\choose 2}-1\choose T}2^{1-{t\choose 2}}.
    \end{align}
Let $G$ be an $(\alpha,t)$-\textit{good} graph on $n\geq n_1$ vertices and $\cQ_G= \cQ_G^{\alpha,t}$ be as in Definition \ref{def: Q_G}, i.e., the collection of subsets of $V(G)$ that induce $\alpha$-quasicliques on $t$ vertices. By the definition of $(\alpha,t)$\textit{-good}, we have that for every $e$ in $E(G)$ and $f$ in $E(\overline{G})$,
    \begin{align}
        \label{eqn: |Q_G}
        |\cQ_G| &= (1\pm n^{-A\alpha^2})N_0,\\
        \label{eqn: |Q_G(e)|}
        |\cQ_G(e)| &= (1\pm n^{-A\alpha^2})N_1,\\
        \label{eqn: |Q_G(f)}
        |\cQ_G(f)| &= (1\pm n^{-A\alpha^2})N_2.   
    \end{align}
    The proof of Lemma \ref{thm: upperbound} follows two steps. First, we will choose a $\cC$ that satisfies some convenient properties (stated in Proposition \ref{thm: upperboundapproximation}) so that it is ``approximately" a $k$-cover, and then we will make alterations to $\cC$ so that we obtain a $k$-cover. 
    \\
    Given a collection $\cC$ of subsets of $V(G)$, let:
    \begin{align}
    X(\cC) &= |\{ e\in E(G): |\cC(e)|< k \}|,\nonumber\\
    \label{eqn: defXYZ}
    Y(\cC) &= |\{f\in E(\overline{G}): |\cC(f)|\geq k\}|,\\
    Z(\cC) &= |\{f \in E(\overline{G}): |\cC(f)|>k\log n \}|.\nonumber
    \end{align}
    \begin{proposition}
    \label{thm: upperboundapproximation}
        There exists a $\cC$ such that, 
        \begin{align}
    \label{eqn: X}
    X(\cC) &< 4n^2\exp\left(-\frac{k^{2\epsilon}}{27}\right),\\
    \label{eqn: Y}
    Y(\cC) &< 4n^2\exp\left(-\frac{k^{2\epsilon}}{6}\right),\\
    \label{eqn: Z}
    Z(\cC) &= 0,\\
    \label{eqn: FinalcC}
    |\cC| &= \left(1\pm \frac{2}{3}\right)\frac{n^2}{c^2k^{1-4\epsilon}\log^2 n}. 
\end{align}
    \end{proposition}
    \begin{proof}
    Let $\cC$ be a random subset of $\cQ_G$ formed by choosing each element in $\cQ_G$ independently with probability $q$. We will set $q = k(1+\alpha)/N_1 < 1$. 
    \\
    In what follows, we will show that the value of $q$ is such that $\EE[|\cC(e)|] > k$ and $\EE[\cC(f)]< k$ for every edge $e$ and every non-edge $f$ in $G$, respectively. We will observe the following inequality, which we use in subsequent calculations.  
    \begin{align*}
        n^{-A\alpha^2} < \alpha/\log n.
    \end{align*}
    Indeed, in view of (\ref{sec4values}) and (\ref{eqn: alphatink}),  $k\leq \log n$ and $\alpha = k^{-1/2 + \epsilon}$, and we have, 
    \begin{align*}
        -A\alpha^2\log n \leq -A(\log n)^{2\epsilon} \leq \log \alpha -\log\log n,
    \end{align*}
    which implies the inequality. 
    \\
    Fix $e$ in $E(G)$ and $f$ in $E(\overline{G})$. Then $|\cC|, |\cC(e)|$ and $|\cC(f)|$ are random variables with binomial distribution. Since $T = (\frac{1}{2}+\alpha){t\choose 2}$, by (\ref{eqn: n_0}), (\ref{eqn: n_1}), and (\ref{eqn: n_2}), for large enough $n$, their expected values are: 
    \begin{align}
        \label{eqn: |cC|}
        \EE[|\cC|] &= q|\cQ_G| = (1\pm n^{-A\alpha^2})\frac{k(1+\alpha)N_0}{N_1}\nonumber\\
        &= (1\pm n^{-A\alpha^2})k(1+\alpha)\frac{n(n-1)}{t(t-1)}\frac{{t\choose 2}}{2T}\nonumber\\
        &= (1\pm n^{-A\alpha^2})\frac{1+\alpha}{1+2\alpha}\cdot k\frac{n(n-1)}{t(t-1)}\nonumber\\
        &= (1\pm 2\alpha) k\frac{n^2}{t^2}.
    \end{align}

    \begin{align}
        \EE[|\cC(e)|] &= q|\cQ_G(e)| = (1\pm n^{-A\alpha^2})k(1+\alpha)\nonumber\\
        \label{eqn: C(e)}
        &= \left(1+\alpha \pm \frac{\alpha}{2}\right)k.
    \end{align}

    \begin{align}
        \EE[|\cC(f)|] &= q|\cQ_G(f)| = (1\pm n^{-A\alpha^2})\frac{k(1+\alpha)N_2}{N_1}\nonumber\\
        &= (1\pm n^{-A\alpha^2}) k(1+\alpha) \frac{{t\choose 2}-T}{T}\nonumber\\
        &= (1\pm n^{-A\alpha^2}) k(1+\alpha) \frac{1-2\alpha}{1+2\alpha}\nonumber\\
        \label{eqn: C(f)}
        &= (1-3\alpha \pm \alpha)k.  
    \end{align}
Now we consider the random variables $X(\cC), Y(\cC), Z(\cC)$. We will now show that there is a choice of $\cC$ such that $X,Y,Z$ are all small. From (\ref{eqn: C(e)}) and using Chernoff Bound (\ref{eqn: Chernoff1}), we obtain:
\begin{align}
    \EE[X] &= \sum_{e\in E(G)} \PP(|\cC(e)|< k)\nonumber\\
    &< \sum_{e\in E(G)} \PP\left(|\cC(e)|< \left(1-\frac{\alpha}{3}\right)\EE[\cC(e)]\right) < \frac{n^2}{2}\cdot 2\exp\left(-\frac{\alpha^2}{27}k\right)\nonumber\\ 
    \label{eqn: E[X]}
    &= n^2\exp\left(-\frac{k^{2\epsilon}}{27}\right). 
\end{align}
Similarly, from (\ref{eqn: C(f)}) and (\ref{eqn: Chernoff1}), we have: 
\begin{align}
    \EE[Y] &= \sum_{f\in E(\overline{G})} \PP(|\cC(f)|\geq k)\nonumber\\
    &< \sum_{f\in E(\overline{G})} \PP\left(|\cC(f)|< \left(1+\alpha\right)\EE[|\cC(f)|]\right) < \frac{n^2}{2}\cdot 2\exp\left(-\frac{\alpha^2}{6}k\right)\nonumber\\
    \label{eqn: E[Y]}
    &= n^2\exp\left(-\frac{k^{2\epsilon}}{6}\right). 
\end{align}
Also, by (\ref{eqn: Chernoff2}), 
\begin{align}
\label{eqn: E[Z]}
    \EE[Z] &= \sum_{f\in E(\overline{G})} \PP(|\cC(f)|\geq k\log n) < \frac{n^2}{2}\cdot \exp\left(-k\log n\right)=o(1). 
\end{align}
Using Markov's Inequality and in view of (\ref{eqn: E[X]}), (\ref{eqn: E[Y]}), and (\ref{eqn: E[Z]}), we have that: 
\begin{align}
\label{eqn: PX}
    \PP\left(X\geq 4n^2\exp\left(-\frac{k^{2\epsilon}}{27}\right)\right)&<\PP( X\geq 4\EE[X])< 1/4,\\
\label{eqn: PY}
    \PP\left(Y\geq 4n^2\exp\left(-\frac{k^{2\epsilon}}{6}\right)\right)&<\PP(Y \geq 4\EE[Y]) < 1/4,\\
\label{eqn: PZ} 
    \PP(Z\geq 1) &<\EE[Z] = o(1). 
\end{align}
Lastly, from (\ref{sec4values}) we have that $t= c\alpha^{-2}\log n \leq ck\log n$, and $k\leq (\log n)^{1/2}$. Consequently, in view of (\ref{eqn: |cC|}), we have that $\EE[|\cC|]\geq n$. Since, $|\cC|$ is a binomial random variable, by using Chernoff Bound (\ref{eqn: Chernoff1}), we have that: 
\begin{align}
\label{eqn: PcC}
 \PP\left(||\cC|-\EE[|\cC|]|>\frac{\EE[|\cC|]}{2}\right) < \exp\left(-\frac{\EE[|\cC|]}{12}\right)< \exp(-n/12).
\end{align}
To summarise, in view of (\ref{eqn: PX}), (\ref{eqn: PY}) and (\ref{eqn: PZ}) and (\ref{eqn: PcC}), with positive probability, one can find $\cC$ satisfying (\ref{eqn: X}), (\ref{eqn: Y}) and (\ref{eqn: Z}) and this proves the proposition. 
\end{proof}
We will now form the multiset $\widetilde{\cC}$ from the $\cC$ given by Proposition \ref{thm: upperboundapproximation} by the following two steps: 
\begin{enumerate}
    \item [Step 1:] For every non-edge $f$, with $|\cC(f)|\geq k$, we remove from $\cC$ the set $\cC(f)$ of elements that covers $f$. 
    \\
    In view of (\ref{eqn: Z}) and (\ref{eqn: defXYZ}), $|\cC(f)|\leq k\log n$ for every non-edge in $G$. Therefore, we remove at most $\left| \bigcup_{f:|\cC(f)|\geq k} \cC(f) \right| \leq Y\cdot k\log n$ elements of $\cC$.  Let the new collection be 
    $$\cC_1 := \cC\setminus \bigcup_{f:|\cC(f)|\geq k} \cC(f).$$
    \item [Step 2:] For every edge $e$, such that $|\cC_1(e)|<k$, we add $k$ copies of the pairs of vertices of $e$ in $\cC_1$ to form $\widetilde{\cC}$. 
    \\
    If $E_2$ is the number of edges $e$, such that $|\cC_1(e)|<k$, then $|\widetilde{\cC}|< |\cC|+ k\cdot E_2$.  
\end{enumerate}
Clearly $\widetilde{\cC}$ obtained after the above steps is a $k$-cover. It remains to bound $|\widetilde{\cC}|$.
\\
Let $E_1$ be the number of edges covered by elements of $\cC\setminus \cC_1$. Each such vertex subset of $\cC\setminus \cC_1$ can cover at most $t^2$ edges. Together with the bound on $|\cC\setminus \cC_1|$  in Step 1, this implies,   
\begin{align*}
    E_1  <t^2 |\cC\setminus \cC_1|\leq t^2\cdot Y\cdot k\log n
\end{align*}
In view of (\ref{eqn: Y}) and since $t = c\alpha^{-2}\log n$, $\alpha = k^{-1/2+\epsilon}$,
\begin{align}
\label{eqn: E1step}
    E_1 < c^2 k^{3-4\epsilon}(\log n)^3 Y < c^2 k^{3}(\log n)^3\cdot 4 n^2\exp\left(-\frac{k^{2\epsilon}}{6}\right)
\end{align}
Now we will give an upper bound on $E_2$. An edge $e$ has $|\cC_1(e)|< k$ in the following cases. Either $|\cC(e)|< k$, and $X=X(\cC)$ such edges or $e$ was covered by some elements of $\cC\setminus \cC_1$ and there are $E_1$ such edges. Thus, in view of (\ref{eqn: X}), we have,
\begin{align}
\label{eqn: E2}
    E_2 &< X + E_1 < 4n^2\exp\left(-\frac{k^{2\epsilon}}{27}\right) + E_1.
\end{align}
Since $ (\log\log n)^{1/\epsilon}\leq k \leq (\log n)^{1/2}$, for large enough $n$,
\begin{align*}
\label{eqn: exp(-k)calc}
  \exp\left(-\frac{k^{2\epsilon}}{6}\right)\leq  \exp\left(-\frac{k^{2\epsilon}}{27}\right) \leq  \exp\left(-\frac{(\log \log n) ^{2}}{27}\right) = \left(\frac{1}{\log n}\right)^{(\log\log n)/27}.
\end{align*}
Consequently, in view of (\ref{eqn: E1step}) and (\ref{eqn: E2}), we have:
\begin{align}
    E_2 < \frac{n^2}{(\log n)^{(\log\log n)/30}}.
\end{align}
We need to add $k\cdot E_2$ edges to $\cC_1$ in Step 2. In view of (\ref{sec4values}), $k\leq \log n$ and hence, $k\cdot E_2 = o(|\cC|)$. Consequently, in view of (\ref{eqn: FinalcC}), 
\begin{align*}
    |\widetilde{\cC}| &< |\cC|+k\cdot E_2 = (1+o(1))|\cC|< \frac{c_2n^2}{k^{1-4\epsilon}\log^2 n},
\end{align*}
for an appropriately chosen constant $c_2>0$.

\section{Proof of Counting Lemma}
\label{sec: sec5}
In this section we will prove Lemma \ref{thm: counting}. Throughout this section, $c$ and $A$ are absolute constants with the value: 
\begin{align*}
    c = \frac{1}{100\log 2}; \qquad A =10^{-4}c;
\end{align*}
Recall that Lemma \ref{thm: counting} states that for every $\epsilon >0$, there exists $n_1:=n_1(\epsilon)$ such that whenever integers $n,t$ and $\alpha>0$ satisfy:
\begin{align}
\label{eqn: 2.6condition}
    n\geq n_1; \qquad \left(\frac{1}{\log n}\right)^{1/2-\epsilon} \leq \alpha \leq \frac{1}{2}; \qquad t = c\alpha^{-2}\log n,
\end{align}
$G(n,1/2)$ is $(\alpha,t)$-\textit{good} with probability $1-2\exp(-n^{1/5})$. 
\subsection{Organisation of Proof:} The rest of the section is organised as follows. Sublemma \ref{thm: PropertyP} and Sublemma \ref{thm: counting2} together imply Lemma \ref{thm: counting}. We state them in the next two subsections.
\\
Sublemma \ref{thm: PropertyP} is proved in Section \ref{sec: sec6}. The proof of Sublemma \ref{thm: counting2} requires Propositions \ref{thm: numberofgammatuples}, \ref{thm: martingalelemma}, and \ref{thm: multiplicity}. In Subsection \ref{sec: proofofcountingsublemma}, we state these 3 propositions and prove Sublemma \ref{thm: counting2} using them. Propositions \ref{thm: numberofgammatuples}, \ref{thm: martingalelemma} and \ref{thm: multiplicity} are proved in Section \ref{sec: sec7}.

\subsection{Extension Property} Let $G$ be a graph. For $S\subseteq V(G)$, a vertex $x\in S$ and integer $j$ such that $0\leq j\leq |S|$, let:
\begin{align*}
    W_S^j &:= \{v\in V(G)\setminus S: |N(v)\cap S| = j\},\\
    W_{S,x}^j &:= \{v\in W_S^j: x\in N(v) \}.
\end{align*}
Observe that, 
\begin{equation}
    \label{eqn: countjstars}
    j|W_S^j| = \sum_{x\in S}|W_{S,x}^j|.
\end{equation}
For $S\subseteq V(G(n,1/2))$, such that $|S|=s$, $x\in S$, and $0\leq j\leq s$, the expectations of the random variables $|W_S^j|$ and $|W_{S,x}^j|$ are as follow.  
\begin{align}
    \label{eqn: expjstarcount}
    \omega_s^j := \EE[|W_{S,x}^j|] &= (n-s){s-1\choose j-1}2^{-s}\\
    \label{eqn: expjstarcount1}
    \mu_s^j := \EE[|W_{S}^j|] &= (n-s){s\choose j}2^{-s}
\end{align}
\begin{definition}[Property $\sP(\alpha,t)$]
\label{def: PropertyP}
    Given $\alpha >0$ and integer $t$, a graph $G$ on $n$ vertices satisfies property $\sP(\alpha,t)$ if for every pair of integers $(s,j)$ that satisfy
    \begin{enumerate}
        \item $0\leq s\leq c\log n$ and $0\leq j\leq s$ or,
        \item $c\log n\leq s\leq t-1$ and $\frac{s}{2}\leq j\leq (\frac{1}{2}+2\alpha)s$, 
    \end{enumerate}
    and for every $S\subseteq V(G)$ with size $s$ and every $x\in S$, we have:
    \begin{align*}
        |W_{S}^j| = (1\pm n^{-1/5})\mu_s^j; \qquad |W_{S,x}^j|= (1\pm n^{-1/5})\omega_s^j.
    \end{align*}
\end{definition}
The following sublemma shows that $G(n,1/2)$ satisfies such a property with high probability.
\begin{sublemma}
\label{thm: PropertyP}
    For every $\epsilon > 0$, there exists $n_1(\epsilon)$ such that if $n,\alpha,t$ satisfy (\ref{eqn: 2.6condition}), then $G(n,1/2)$ satisfies $\sP(\alpha,t)$ with probability $1-2\exp(-n^{1/5})$. 
\end{sublemma}
\subsection{Counting Quasicliques} We will now count the number of $\alpha$-quasicliques on $t$ vertices that contain a subset $U\subseteq V(G)$ for a graph $G$ satisfying $\cP(\alpha,t)$. 
\label{sec: proofofcountingsublemma}
\begin{sublemma}
\label{thm: counting2}
    For every $\epsilon > 0$, there exists $n_1(\epsilon)$ such that if $n,\alpha,t$ satisfy (\ref{eqn: 2.6condition}), and $G$ is a graph on $n$ vertices satisfies $\sP(\alpha,t)$, then for every subset $U\subseteq V(G)$, with $|U|=l\in \{0,2\}$. 
    \begin{align*}
        |\cQ_G(U)| = (1\pm n^{-A\alpha^2}){n-l\choose t-l} {{{t\choose 2}-{l\choose 2}} \choose T-e(U)} 2^{{l\choose 2}-{t\choose 2}}.
    \end{align*}
\end{sublemma}
Note that Sublemma \ref{thm: PropertyP} and \ref{thm: counting2} immediately imply Lemma \ref{thm: counting}. As stated before, now we will state 3 propositions (which are proved in Section \ref{sec: sec7}) and use them to prove Sublemma \ref{thm: counting2}.
\\[0.2cm]
For the rest of the subsection, we fix $\epsilon> 0$, and let $n_1 = n_1(\epsilon)$ be an integer. Further we fix $n,\alpha, t$ satisfying (\ref{eqn: 2.6condition}). Let $G$ be a graph on $n$ vertices that satisfies $\sP(\alpha,t)$. For the rest of the section, we denote $\cQ=\cQ_G$. Fix $U\subseteq V(G)$ with $U=\{u_1,\dots, u_l\}$.

Our plan is to count $|\cQ(U)|$ by extending $U$ sequentially to form an ordered tuple of $t$ vertices where each vertex sends roughly $(\frac{1}{2}+\alpha)$ proportion of edges backwards. To formalise this, we make the following definitions. 
\begin{definition}[Backward Degree Sequences]
\label{def: gammasequences}
    Given $\gamma>0$, let $\cJ_\gamma(U)$ be the collection of tuples $(j_l,\dots, j_{t-1})$ such that 
    \begin{enumerate}
        \item $0\leq j_s\leq s$ for all $l\leq s\leq t-1.$
        \item $j_s = \left(\frac{1}{2}+\alpha\pm \gamma\right)s$ for all $c\log n \leq s\leq t-1$. 
        \item $j_l + \cdots + j_{t-1}= T-e(U)$.
    \end{enumerate}
\end{definition}
Observe that if $G$ satisfies $\sP(\alpha,t)$, then for any $\bj \in \cJ(U)$ and any $s$ such that $l\leq s\leq t-1$, the pair of integers $(s,j_s)$ satisfy either of $(1)$ or $(2)$ in Definition \ref{def: PropertyP}. 
\begin{definition}
    Given $\bj\in \cJ_\gamma(U)$, define $\overrightarrow{\cB}_\gamma(U;\bj)$ as the collection of $t$-tulpes of $V(G)$, $(x_1,\dots,x_t)$ such that
    \begin{enumerate}
        \item $x_1 = u_1,\cdots, x_l = u_l$.
        \item $x_{s+1}\in W_{\{x_1,\dots, x_s\}}^{j_s}$ for every $l \leq s\leq t-1$, i.e., $x_{s+1}$ sends $j_s$ edges backwards to $\{x_1,\dots, x_s\}$. 
    \end{enumerate}
\end{definition}
We define $\overrightarrow{\cB}_\gamma(U)$ to be the disjoint union of all $\overrightarrow{\cB}_\gamma(U;\bj)$, i.e.,
\begin{align*}
    \overrightarrow{\cB}_\gamma(U):= \bigcup_{\bj\in \cJ_\gamma(U)}\overrightarrow{\cB}_\gamma(U;\bj). 
\end{align*}
A tuple $(x_1,\dots, x_t)$ is called $\gamma$-admissible if it is in $\overrightarrow{\cB}_\gamma(U)$. 
\begin{proposition}
    \label{thm: numberofgammatuples}
    For every $\frac{\alpha}{10}\leq \gamma\leq \alpha$,
    \begin{align*}
        |\overrightarrow{\cB}_\gamma(U)| = (1\pm n^{-3A\alpha^2})(t-l)!{n-l\choose t-l} {{{t\choose 2}-{l\choose 2}} \choose T-e(U)} 2^{{l\choose 2}-{t\choose 2}}.
    \end{align*}
\end{proposition}
Note that if for some $\gamma$ in the above range, every $(x_1,\dots, x_t)\in \overrightarrow{\cB}_\gamma(U)$ was such that $\{x_1,\dots, x_t\}$ is an $\alpha$-quasiclique, and every $\{x_1,\dots, x_t\}$ in $\cQ(U)$ contributed $(t-l)!$ tuples to $\overrightarrow{\cB}_\gamma(U)$, we would have,
\begin{align*}
    (t-l)!|\cQ(U)| = |\overrightarrow{\cB}_\gamma(U)|
\end{align*}
and Sublemma \ref{thm: counting2} would immediately follow. However, this is not the case. Instead, we prove the following claim which together with Proposition \ref{thm: numberofgammatuples} implies Sublemma \ref{thm: counting2}.
\begin{claim} We have,
\label{ref: claim6.3}
    \begin{align*}
  (1- n^{-2A\alpha^2})|\overrightarrow{\cB}_{\frac{\alpha}{10}}(U)|\leq  (t-l)!|\cQ(U)|\leq (1+ n^{-2A\alpha^2})|\overrightarrow{\cB}_\alpha(U)|
\end{align*}
\end{claim}
In what follows, we prove the above claim and Sublemma \ref{thm: counting2}. We define 
\begin{align*}
    \overrightarrow{\cQ}_\gamma(U):= \overrightarrow{\cB}_\gamma(U)\cap\{(x_1,\dots, x_t): \{x_1,\dots, x_t\}\in \cQ(U) \}
\end{align*}
Observe that, 
\begin{align}
\label{eqn: cQincB}
    \overrightarrow{\cQ}_\gamma(U)\subseteq \overrightarrow{\cB}_\gamma(U).
\end{align}
We show that for a sufficiently small $\gamma$ (in this case $\gamma = \frac{\alpha}{10})$, $\overrightarrow{\cQ}_\gamma(U)$ forms the majority of $\overrightarrow{\cB}_\gamma(U)$.
\begin{proposition}
\label{thm: martingalelemma}
    For $\gamma = \frac{\alpha}{10}$,
    \begin{align*}
        |\overrightarrow{\cQ}_\gamma(U)| = (1\pm n^{-A})|\overrightarrow{\cB}_\gamma(U)|.
    \end{align*}
\end{proposition}

Finally, given $\gamma>0$, for each element of $\cQ(U)$, one can have at most $(t-l)!$ ordered tuples corresponding to it in $\overrightarrow{\cQ}_\gamma(U)$. However, it is possible that some of the $(t-l)!$ orderings do not form $\gamma$-admissible tuples. The following proposition shows that if one allows $\gamma$ to be large enough, (in this case $\gamma=\alpha$ is sufficient), then most of the orderings of an element in $\cQ(U)$ has the right backward degrees, i.e., they are in $\overrightarrow{\cQ}_\gamma(U)$.
\begin{proposition}
    For $\gamma=\alpha$,
    \label{thm: multiplicity}
    \begin{align*}
        |\overrightarrow{\cQ}_\gamma(U)|= (1\pm n^{-3A\alpha^2})(t-l)!|\cQ(U)|.
    \end{align*}
\end{proposition}
\begin{proof}[Proof of Claim \ref{ref: claim6.3}]
  Note that since $\cJ_{\frac{\alpha}{10}}(U)\subseteq \cJ_{\alpha}(U)$, we have that $\overrightarrow{\cQ}_{\frac{\alpha}{10}}(U)\subseteq \overrightarrow{\cQ}_{\alpha}(U)$. Also in view of (\ref{eqn: cQincB}), we have:
    \begin{align}
    \label{eqn: 5.5.1}
        \overrightarrow{\cQ}_{\frac{\alpha}{10}}(U)\subseteq \overrightarrow{\cQ}_{\alpha}(U)\subseteq \overrightarrow{\cB}_{\alpha}(U)
    \end{align}
    From Prop. \ref{thm: martingalelemma} and Prop. \ref{thm: multiplicity} we know,
    \begin{align}
        \label{eqn: 5.5.2}
        |\overrightarrow{\cQ}_{\frac{\alpha}{10}}(U)| = (1\pm n^{-A})|\overrightarrow{\cB}_{\frac{\alpha}{10}}(U)|,
    \end{align}
    and,
    \begin{align}
    \label{eqn: 5.5.3}
        |\overrightarrow{\cQ}_{\alpha}(U)| = (1\pm n^{-3A\alpha^2}) (t-l)! |\cQ(U)|.
    \end{align}
Consequently, in view of (\ref{eqn: 5.5.3}):
\begin{align*}
    |\cQ(U)|\geq \frac{1}{(1+ n^{-3A\alpha^2})}\frac{1}{(t-l)!} |\overrightarrow{\cQ}_{\alpha}(U)|\geq (1-n^{-3A\alpha^2})\frac{|\overrightarrow{\cQ}_{\frac{\alpha}{10}}(U)|}{(t-l)!},
\end{align*}
And similarly, 
\begin{align*}
    |\cQ(U)| \leq \frac{1}{(1 - n^{-3A\alpha^2})}\frac{|\overrightarrow{\cQ}_{\alpha}(U)|}{(t-l)!}\leq  (1+ n^{-2A\alpha^2})\frac{|\overrightarrow{\cB}_{\alpha}(U)|}{(t-l)!}.
\end{align*}
In view of (\ref{eqn: 5.5.2}), we conclude that, 
\begin{align*}
  (1- n^{-2A\alpha^2})\frac{|\overrightarrow{\cB}_{\frac{\alpha}{10}}(U)|}{(t-l)!}  \leq |\cQ(U)| \leq (1 + n^{-2A\alpha^2})\frac{|\overrightarrow{\cB}_{\alpha}(U)|}{(t-l)!}.
\end{align*}
\end{proof}
\begin{proof}[Proof of Sublemma \ref{thm: counting2}]
From Prop. \ref{thm: numberofgammatuples}, we have that for $\frac{\alpha}{10}\leq \gamma\leq \alpha$
    \begin{align}
    \label{eqn: 5.5.4}
        |\overrightarrow{\cB}_\gamma(U)| = (1\pm n^{-3A\alpha^2})(t-l)!{n-l\choose t-l}{{{t\choose 2}-{l\choose 2}}\choose T-e(U)}{2^{{l\choose 2}-{t\choose 2}}},
    \end{align}
Together with Claim \ref{ref: claim6.3}, this implies,     
\begin{align}
    |\cQ(U)| = (1 \pm n^{-A\alpha^2}){n-l\choose t-l}{{{t\choose 2}-{l\choose 2}}\choose T-e(U)}{2^{{l\choose 2}-{t\choose 2}}}.
\end{align}
\end{proof}

\section{Proof of Sublemma \ref{thm: PropertyP}}
\label{sec: sec6}
Fix $\epsilon > 0$, and let $n,t, \alpha$ such that $n,t$ are integers and they satisfy (\ref{eqn: 2.6condition}), i.e., 
\begin{align*}
     n\geq n_1; \qquad \left(\frac{1}{\log n}\right)^{\frac{1}{2}-\epsilon} \leq \alpha \leq \frac{1}{2}; \qquad t = c\alpha^{-2}\log n.
\end{align*}
Whenever necessary, we will use that $n_1$ is large enough. Let $V = V(G(n,1/2))$. If $S\subseteq V$ with $|S|=s$ and $0\leq j \leq s$. Then for any $x\in S$,
\begin{align*}
    \EE[|W_S^j|] = (n-s)\frac{{s\choose j}}{2^s} = \mu_s^j; \qquad 
    \EE[|W_{S,x}^j|] = (n-s)\frac{{s-1\choose j-1}}{2^s} = \omega_s^j.
\end{align*}
To show that $G(n,1/2)$ satisfies $\sP(\alpha,t)$ with high probability, we will show that for every pair of integers $(s,j)$ satisfying:
\begin{enumerate}
        \item $0\leq s< c\log n$ and $0\leq j\leq s$ or,
        \item $c\log n\leq s\leq t-1$ and $\frac{s}{2}\leq j\leq (\frac{1}{2}+2\alpha)s$, 
\end{enumerate}
and for every $S\subseteq V$ and every $x\in S$, 
\begin{align*}
        |W_{S}^j| = (1\pm n^{-1/5})\mu_s^j; \qquad |W_{S,x}^j|= (1\pm n^{-1/5})\omega_s^j.
\end{align*}
with probability $1-2\exp(-n^{1/5})$.
\\
In what follows, we will need to use the following claim. We defer its proof to the end of the section. 
\begin{claim}
\label{thm: binomialinequality}
    If $(s,j)$ are integers that satisfy:
    \begin{enumerate}
        \item[$(1')$] $1\leq s\leq c\log n$ and $1\leq j\leq s$ or,
        \item[$(2')$] $c\log n\leq s\leq t-1$ and $\frac{s}{2}\leq j\leq (\frac{1}{2}+2\alpha)s$, 
    \end{enumerate}
    then we have that,
    \begin{align*}
        \frac{{s-1\choose j-1}}{2^s} \geq n^{-1/5}.
    \end{align*}
\end{claim}
We will now prove the following Proposition and then prove Sublemma \ref{thm: PropertyP} using it. 
\begin{proposition}
\label{thm: PropertyP2}
    The probability that there exists a pair of integers $(s,j)$ satisfying (1') or (2'), and $S\subseteq V$ with $|S|=s$ and $x\in S$, such that, 
\begin{align*}
    ||W_{S,x}^j| - \omega_{s}^j|> n^{-1/5}\omega_{s}^j,
\end{align*}
is $\exp(-n^{1/5})$.
\end{proposition}
\begin{proof}
    Fix a pair of integers $(s,j)$ satisfying (1') or (2') and fix $S\subseteq V$ and $x\in S$ such that $|S|=s$. Note that $|W_{S,x}^j|$ is a binomial random variable. Thus by Chernoff bounds (\ref{eqn: Chernoff1}), with $\epsilon = n^{-1/5}$, we have that:
\begin{align*}
    \PP(||W_{S,x}^j| - \omega_s^j|> n^{-1/5}\omega_s^j) < 2\exp\left(-n^{-2/5}\frac{(n-s){s-1\choose j-1}}{2^s}\right)
\end{align*}
Note that $s\leq t-1 < c\alpha^{-2}\log n$. By (\ref{eqn: 2.6condition}), $\alpha^{-2}\leq \log n$, and hence $s \leq c\log^2 n$. Together with Claim \ref{thm: binomialinequality}, this implies
\begin{align*}
    \PP(||W_{S,x}^j| - \omega_s^j|> n^{-1/5}\omega_s^j) < 2\exp\left(-n^{-2/5}\frac{(n-s){s-1\choose j-1}}{2^s}\right) &\leq  2\exp\left(-n^{-3/5}(n-s)\right)\\ &< 2\exp\left(-\frac{n^{2/5}}{2}\right).
\end{align*}
Thus the probability that there exists a pair of integers $(s,j)$ satisfying (1') or (2'), and $S\subseteq V$ with $|S|=s$ and $x\in S$, such that, 
\begin{align*}
    ||W_{S,x}^j| - \omega_s^j|> n^{-1/5}\omega_s^j,
\end{align*}
is at most, 
\begin{align}
\label{eqn: Prop7.3}
    \sum_{s=0}^{t-1} \sum_{j=0}^{s}  \sum_{S\subseteq V}\sum_{x\in S} \quad \PP(||W_{S,x}^j| - \omega_s^j|> n^{-1/5}\omega_s^j) &< t\cdot t\cdot {n\choose t}\cdot t \cdot 2\exp\left(-\frac{n^{2/5}}{2}\right).
\end{align}
As mentioned before, from (\ref{eqn: 2.6condition}), $\alpha^{-2}\leq \log n$ and $t =c\alpha^{-2}\log n < c\log^2 n$. Consequently, the right hand side of (\ref{eqn: Prop7.3}) is at most $\exp(-n^{1/5})$, and this completes the proof of the proposition. 
\end{proof}

\begin{proof}[Proof of Sublemma \ref{thm: PropertyP}]
Proposition \ref{thm: PropertyP2} implies that whenever $(s,j)$ satisfies (1) or (2), and $S\subseteq V$ with $|S|=s$ and $x\in S$,
\begin{align}
\label{eqn: 7.3B}
    |W_{S,x}^j| = (1\pm n^{-1/5})\omega_s^j.
\end{align}
To complete the proof of Sublemma \ref{thm: PropertyP}, we need to show that for any pair of integers $(s,j)$ satisfying (1) or (2), and any $S\subseteq V$ and $|S|=s$,
\begin{align*}
    |W_{S}^j| = (1\pm n^{-1/5})\mu_{s}^j,
\end{align*}
with probability $1-o(1)$. 
\\
When $s\neq 0$ and $j\neq 0$, this is implied by Proposition \ref{thm: PropertyP2}. Indeed, in view of (\ref{eqn: expjstarcount}) and (\ref{eqn: 7.3B}),  
\begin{align*}
    |W_S^j| = \frac{1}{j}\sum_{x\in S}|W_{S,x}^j| = \frac{s}{j}(1\pm n^{-1/5})\omega_s^j =(1\pm n^{-1/5})\mu_{s}^j.
\end{align*}
Now, to complete the proof of Sublemma \ref{thm: PropertyP}, we need to address the case where $s= 0$ or $j=0$. For this case, one can show the following. If $s=0$ or $j=0$, and $S\subseteq V$ with $|S|=s$, then
\begin{align}
\label{eqn: 69}
    |W_{S}^j| = (1\pm n^{-1/5})\mu_{s}^j,
\end{align}
with probability at least $1-\exp(-n^{1/5})$. We omit the proof of (\ref{eqn: 69}) since it is a standard application of Chernoff bounds (\ref{eqn: Chernoff1}) followed by a union bound. 
\end{proof}

\begin{proof}[Proof of Claim \ref{thm: binomialinequality}]
    Recall that $c= \frac{1}{100\log 2}$. If $(s,j)$ are a pair of integers such that they satisfy $(1')$, i.e., $1\leq s\leq c\log n$ and $1\leq j\leq s$, then we have that, 
    \begin{align*}
        \frac{{s-1\choose j-1}}{2^s} \geq 2^{-s} \geq n^{-c\log 2} \geq n^{-1/5}.
    \end{align*}
    For the second case, let $(s,j)$ be a pair of integers that satisfy $(2')$, i.e., $c\log n\leq s\leq t-1$ and $\frac{1}{2}s\leq j\leq (\frac{1}{2}+2\alpha)s$. Let $j = (\frac{1}{2}+\beta)s$ where $0\leq \beta\leq 2\alpha$. \\
     We will use the following two facts. The first fact is mentioned in \cite{Cover2006}. 
    \begin{fact}
    \label{thm: entropyfact}
        For $0< p < 1$, let $H(p):= -p\log_2 p - (1-p)\log_2(1-p)$. We have that,
        \begin{align*}
            \frac{{s\choose j}}{2^s}\geq \frac{1}{s+1}2^{s(H(j/s)-1)}.
        \end{align*}
    \end{fact}
    \begin{proof}
    We have that,
    \begin{align*}
          (s+1){s\choose j}2^{sH(j/s)} = (s+1){s\choose j}\left(\frac{j}{s}\right)^j, 
    \end{align*}
    and thus the fact is equivalent to showing,
    \begin{align*}
       {s\choose j}\left(\frac{j}{s}\right)^j \left(1-\frac{j}{s}\right)^{s-j}\geq \frac{1}{s+1}.
    \end{align*}
   which is true since the left hand side is the largest term in the binomial expansion of $(p + (1-p))^s$ where $p=j/s$.  
    \end{proof}
    \begin{fact} 
     \label{thm: entropy2}
    For $0\leq \beta\leq 2\alpha$,
         \begin{align*}
            H\left(\frac{1}{2}+\beta\right) -1 \geq -4\beta^2
        \end{align*}
    \end{fact}
    \begin{proof}
Let us define
\begin{align*}
    f(\beta) &:= H\left(\frac{1}{2}+\beta\right)-1+4\beta^2\\
    &= -\left(\frac{1}{2}+\beta\right)\log_2\left(1+2\beta\right) - \left(\frac{1}{2}-\beta\right)\log_2\left(1-2\beta\right)-4\beta^2
\end{align*}
Observe that $f(0)=0$ and, 
\begin{align*}
    f'(\beta) = -\frac{1}{\log 2}\log\left(1+\frac{4\beta}{1-2\beta}\right)+8\beta
\end{align*}
We have that, 
\begin{align*}
    \log\left(1+\frac{4\beta}{1-2\beta}\right) < \frac{4\beta}{1-2\beta}
\end{align*}
For $\beta> 0$,
\begin{align*}
    f'(\beta) > 4\beta\left(2 - \frac{1}{(1-2\beta)\log 2}\right)>0,
\end{align*}
and consequently, $f(\beta)\geq f(0)=0$.
\end{proof}
Now we will combine the two facts to complete the proof of Claim \ref{thm: binomialinequality}. We have that, 
    \begin{align*}
        \frac{{s-1 \choose j-1}}{2^s} = \frac{j}{s}\frac{{s \choose j}}{2^s}
    \end{align*}
    Since $j/s = (\frac{1}{2}+\beta)$, using Fact \ref{thm: entropyfact} and Fact \ref{thm: entropy2}, we get:
     \begin{align*}
    \frac{{s-1 \choose j-1}}{2^s}& \geq \frac{j}{s}\frac{1}{(s+1)}2^{-4\beta^2s}\\
    &= \left(\frac{1}{2}+\beta\right)\frac{1}{(s+1)}2^{-4\beta^2s}\\
    &\geq \frac{1}{2(s+1)}2^{-4\beta^2 s}
    \end{align*}
    Since $\beta\leq 2\alpha$ and $s < t=c\alpha^{-2}\log n$, we have $4\beta^2 s \leq 16c\log n$, which implies,
    \begin{align*}
    \frac{{s-1 \choose j-1}}{2^s}& \geq \frac{1}{2(s+1)}n^{-16c} \geq n^{-1/5}.
    \end{align*}
\end{proof}

\section{Proofs of Propositions \ref{thm: numberofgammatuples}, \ref{thm: martingalelemma} and \ref{thm: multiplicity}}
\label{sec: sec7}
Fix $\epsilon> 0$, and let $n_1 = n_1(\epsilon)$ be an integer. Further we fix $n,\alpha, t$ satisfying (\ref{eqn: 2.6condition}), i.e., 
\begin{align}
\label{eqn: 2.6condition2}
    n\geq n_1; \qquad \left(\frac{1}{\log n}\right)^{1/2-\epsilon} \leq \alpha \leq \frac{1}{2}; \qquad t = c\alpha^{-2}\log n.
\end{align}

Let $G$ be a graph on $n$ vertices that satisfies $\sP(\alpha,t)$. For the simplicity of presentation, we will only prove the case where $U$ is an edge in the graph, for each of the propositions. The proof is analogous if one assume $U$ is any subset of size at most $2$. For the rest of the section, we denote $\cQ=\cQ_G^{\alpha,t}$. Fix an edge $e \in E(G)$ with $e=\{u_1,u_2\}$. 

\subsection{Proof of Proposition \ref{thm: numberofgammatuples}} Our goal is to show that for any $\gamma$ such that $\frac{\alpha}{10}\leq \gamma \leq \alpha$,
\begin{align*}
    |\overrightarrow{\cB}_\gamma(e)|=(1\pm n^{-3A\alpha^2})(t-2)!{n-2\choose t-2} {{{t\choose 2}-1} \choose T-1} 2^{1-{t\choose 2}}.
\end{align*}
Fix $\gamma$ such that $\frac{\alpha}{10}\leq \gamma\leq \alpha$. Note that by definition, 
\begin{align*}
    |\overrightarrow{\cB}_\gamma(e)| &= \sum_{\bj \in \cJ_\gamma(e)}|\overrightarrow{\cB}_\gamma(e;\bj)|.
\end{align*}
Fix a $\bj \in \cJ_\gamma(e)$. Since $G$ satisfies Property $\sP(\alpha,t)$, for any $s$ such that $2\leq s\leq t-1$, the pair of integers $(s,j_s)$ satisfy either of $(1)$ or $(2)$ in Definition \ref{def: PropertyP}. Thus if $S\subseteq V(G)$ with $|S|=s$, where $2\leq s\leq t-1$, then $|W_S^{j_s}| = (1\pm n^{-1/5})\mu_s^{j_s}$. 
\\
In particular, if $(x_1,\dots, x_s)$ is a tuple of vertices, then there are roughly $\mu_s^{j_s}$ to extend it to a tuple $(x_1,\dots, x_{s+1})$ such that $x_{s+1}\in W_{\{x_1,\dots, x_s\}}^{j_s}$. Consequently, the number of tuples $(x_1,\dots, x_t)$ in $\overrightarrow{\cB}_\gamma(e;\bj)$, i.e., those tuples for which $\bj$ is the backward degree sequence is, 
\begin{align*}
    |\overrightarrow{\cB}_\gamma(e;\bj)| &= \prod_{s=2}^{t-1}(1\pm n^{-1/5})\mu_s^{j_s} = (1\pm n^{-1/6})\prod_{s=2}^{t-1}(n-s){s\choose j_s}2^{-s}\\
    &= (1\pm n^{-1/6})(t-2)!{n-2\choose t-2}{2^{1-{t\choose 2}}}\prod_{s=2}^{t-1}{s\choose j_s}.
\end{align*}
Consequently, we have that, 
\begin{align*}
    |\overrightarrow{\cB}_\gamma(e)| &=  (1\pm n^{-1/6})(t-2)!{n-2\choose t-2}{2^{1-{t\choose 2}}}\sum_{\bj \in \cJ_\gamma(e)}\prod_{s=l}^{t-1}{s\choose j_s}
\end{align*}
The following claim completes the proof. 

\begin{claim}
\label{thm: clm7.1}
    \begin{align*}
        \sum_{\bj \in \cJ_\gamma(e)}\prod_{s=2}^{t-1}{s\choose j_s} = (1\pm n^{-4A \alpha^2}) {{t\choose 2}-1\choose T-1}.
    \end{align*}
\end{claim}

\begin{proof}
    Let $\cT$ be the collection of graphs on vertex set $[t]$ with $T = \left(\frac{1}{2}+\alpha\right){t\choose 2}$ edges such that $\{1,2\}$ is an edge. In other words, $\cT$ is the collection of $T$ uniform subsets of $[t]^{(2)}$, which contain $\{1,2\}$. We have that, 
    \begin{align}
    \label{eqn: T}
        |\cT| = {{t\choose 2}-1\choose T-1}.
    \end{align}
    Given a fixed $\bj$ in $\cJ_\gamma(F)$, let $\cT_\gamma(\bj)$ denote all the graphs in $\cT$ with its backward degrees given by $\bj$, i.e.,
    \begin{align*}
        \cT_\gamma(\bj):= \{ K\in \cT_\gamma: |N_K(s+1)\cap[s]| = j_s \text{ for all } 2\leq s\leq t-1 \}
    \end{align*}
    Given $\bj\in \cJ_\gamma(e)$, every graph in $\cT_\gamma(\bj)$ is formed by sequentially choosing the backward neighborhoods of the vertices from $1$ upto $t$. Given $s$ such that $2\leq s\leq t-1$, there are $j_s$ choices for the neighborhood of $(s+1)$ in $[s]$. This implies, 
    \begin{align*}
        |\cT_\gamma(\bj)| = \prod_{s=2}^{t-1} {s\choose j_s}. 
    \end{align*}
    Let,
    \begin{align*}
        \cT_\gamma = \bigsqcup_{\bj\in \cJ_\gamma(e)} \cT_\gamma(\bj)
    \end{align*}
    And this implies, 
    \begin{align}
    \label{eqn: Tgamma}
        |\cT_\gamma| &= \sum_{\bj \in \cJ_\gamma(e)} |\cT_\gamma(\bj)| =  \sum_{\bj \in \cJ_\gamma(U)}\prod_{s=2}^{t-1}{s\choose j_s}.
    \end{align}
    We will prove the claim by showing that $\cT_\gamma$ is roughly the size of $\cT$. To do this, we will show that if $K$ is chosen uniformly at random from $\cT$, then with high probability, $K$ is in $\cT_\gamma$. 
    \\
    Let $K$ be a graph chosen at random from $\cT$ by choosing a set of $T-1$ pairs uniformly from $[t]^{(2)}\setminus \{1,2\}$. 
    For $s$ such that $2\leq s\leq t-1$, let $j_s=j_s(K)$ be the backward degree of $K$, i.e.,
    \begin{align*}
        j_s:= |\{\{i,s+1\}: 1\leq i\leq s\}\cap K|. 
    \end{align*}
    Observe that,  
    \begin{align*}
        \frac{|\cT_\gamma|}{|\cT|}=\PP( K \in \cT_\gamma) = \PP ((j_2,\cdots, j_{t-1})\in \cJ_\gamma(e))
    \end{align*}
     Fix $s$ such that $2\leq s\leq t-1$. The random variable $j_s$ follows the hypergeometric distribution $H({t\choose 2}-1, T-1,s)$. Thus, we have: 
    \begin{align}
    \label{eqn: expJs}
        \EE[j_s] = \frac{T-1}{{t\choose 2}- 1}s = \left(\frac{1}{2}+\alpha\right)s\cdot\frac{1- \frac{1}{(\frac{1}{2}+\alpha){t\choose 2}}}{1- \frac{1}{{t\choose 2}}}.
    \end{align}
    Since $(n,\alpha,t)$ satisfy (\ref{eqn: 2.6condition2}), we have that for large enough $n_1$,
    \begin{align*}
          \frac{1}{{t\choose 2}}<\frac{1}{t} = \frac{\alpha^2}{c\log n} < \frac{\alpha}{c\log n}.
    \end{align*}
    Since $\alpha/10\leq \gamma\leq \alpha$, for large enough $n$, (\ref{eqn: expJs}) implies: 
    \begin{align*}
        \EE[j_s] = \left(\frac{1}{2}+\alpha \right)s \left(1\pm \frac{\gamma}{2}\right) = \left(\frac{1}{2}+\alpha \pm \frac{\gamma}{2}\right)s .
    \end{align*}
    Since $j_s$ follows the hypergeometric distribution, using Lemma \ref{thm: Chernoff1}, with $\epsilon = \gamma\cdot s/2\EE[j_s]$, we have that:
    \begin{align}
        \PP\left(|j_s - \EE[j_s]| > \frac{\gamma}{2}s\right) < 2 \exp \left(-\frac{\gamma^2 s^2}{12\EE[j_s]}\right)\leq 2 \exp \left(-\frac{\gamma^2 s}{12}\right).
    \end{align}
    Consequently, 
    \begin{align*}
        \PP((j_2,\dots, j_{t-1})\notin \cJ_\gamma(e)) < \sum_{c\log n\leq s\leq t-1} \PP\left(|j_s - \EE[j_s]| > \frac{\gamma}{2}s\right) <  t\cdot 2 \exp \left(-\frac{c\gamma^2}{12}\log n\right)
    \end{align*}
    Since $(n,\alpha,t)$ satisfy (\ref{eqn: 2.6condition2}), and $\gamma \geq \alpha/10$, 
    \begin{align*}
        2t \exp \left(-\frac{c\gamma^2}{12}\log n\right) < \exp \left(-\frac{c\gamma^2}{24}\log n\right) < \exp \left(-\frac{c\alpha^2}{2400}\log n\right) <\exp \left(-\frac{c\alpha^2}{2500}\log n\right).
    \end{align*}
    Consequently, for $A= 10^{-4}c$, we have that:
    \begin{align*}
        \frac{|\cT_\gamma|}{|\cT|} = \PP((j_2,\dots, j_{t-1})\in \cJ_\gamma(e)) \geq 1- n^{-4A\alpha^2}. 
    \end{align*}
    Together with (\ref{eqn: T}) and (\ref{eqn: Tgamma}), this implies the claim. 
\end{proof}

\subsection{Proof of Proposition \ref{thm: martingalelemma}}
We need to show that for $\gamma =\frac{\alpha}{10}$,
\begin{align*}
      |\overrightarrow{\cQ}_\gamma(e)| = (1\pm n^{-A})|\overrightarrow{\cB}_\gamma(e)|.
\end{align*}
Recall that,
\begin{align*}
    \overrightarrow{\cQ}_\gamma(e)\subseteq \overrightarrow{\cB}_\gamma(e).
\end{align*}
Fix $\gamma =\frac{\alpha}{10}$. For $\bj \in \cJ_\gamma(e)$, let 
    \begin{align*}
        \overrightarrow{\cQ}_\gamma(e;\bj) := \overrightarrow{\cQ}_\gamma(e)\cap \overrightarrow{\cB}_\gamma(e;\bj),
    \end{align*}
i.e., those tuples in $\overrightarrow{\cQ}_\gamma(e)$ whose backward degrees correspond to $\bj$. We have,
\begin{align*}
     |\overrightarrow{\cB}_\gamma(e)| &= \sum_{\bj \in \cJ_\gamma(e)}|\overrightarrow{\cB}_\gamma(e;\bj)|.
\end{align*}
and consequently, 
\begin{align*}
    |\overrightarrow{\cQ}_\gamma(e)| &= \sum_{\bj \in \cJ_\gamma(e)}|\overrightarrow{\cQ}_\gamma(e;\bj)|.
\end{align*}
We will show that, for every  $\bj \in \cJ_\gamma(e)$,
\begin{align}
\label{eqn: MartingaleMain}
     |\overrightarrow{\cQ}_\gamma(e;\bj)| = (1\pm n^{-A})|\overrightarrow{\cB}_\gamma(e;\bj)|,
\end{align}
and summing over all $\bj\in \cJ_\gamma(e)$, this implies the proposition. 
\subsubsection{\textbf{Proof Outline}} Our plan to prove (\ref{eqn: MartingaleMain}) is as follows. We will first describe the process that randomly generates a $t$-tuple, $(z_1,\dots, z_t)$ from $\overrightarrow{\cB}_\gamma(e;\bj)$. Then we show that this random process generates tuples in $\overrightarrow{\cB}_\gamma(e;\bj)$ ``almost uniformly", i.e., any fixed $t$-tuple in $\overrightarrow{\cB}_\gamma(e;\bj)$ is chosen with roughly the same probability. Then we show that with high probability, $(z_1,\dots, z_t)$ is in $\overrightarrow{\cQ}_\gamma(e;\bj)$. Finally, we show that together this implies (\ref{eqn: MartingaleMain}). 

\subsubsection{\textbf{Random Process}} Fix $\bj \in \cJ_\gamma(e)$. It will be convenient to use the following notation. For $i$ such that $1\leq i\leq t$, let
\begin{align*}
    \overrightarrow{\cB}^{(i)}_\gamma(e;\bj) := \left\{(x_1,\dots, x_i): x_{s+1}\in W_{\{x_1,\dots, x_s\}}^{j_s} \text{ for every } 0\leq s\leq i-1\right\},
\end{align*}
i.e., it is the collection of truncated tuples of length $i$ in $\overrightarrow{\cB}_\gamma(e;\bj)$. \\
We generate the random tuple $(z_1,\dots, z_t)$ in $\overrightarrow{\cB}_\gamma(e;\bj)$ by the following random process. 
\begin{enumerate}
    \item Fix $z_1 = \oz_1 = u_1$ and $z_2 = \oz_2 = u_2$, i.e., $\{z_1,z_2\} = e$. 
    \item Let $s$ such that $2\leq i\leq t-1$, and suppose $\oz_1, \oz_2, \dots, \oz_i$ are already chosen. The vertex $z_{i+1}$ is chosen uniformly at random from $W= W_{\{\oz_1,\dots, \oz_i\}}^{j_i}$, i.e., $z_{i+1}$ is chosen uniformly at random from those vertices which send $j_i$ edges to the set $\{\oz_1,\dots, \oz_i\}$, i.e.,
    \begin{align}
    \label{ref: randomprocess}
        \PP(z_{i+1} = \oz_{i+1}| z_1 = \oz_1 \dots, z_i = \oz_i) = \frac{1}{|W_{\{\oz_1,\dots, \oz_i\}}^{j_i}|}.
    \end{align}
    for any $\oz_{i+1}\in W_{\{\oz_1,\dots, \oz_i\}}^{j_i}$. 
\end{enumerate}
Observe that the truncated random tuple $(z_1,\dots, z_i)$ is in $\overrightarrow{\cB}^{(i)}_\gamma(e;\bj)$. 
\subsubsection{\textbf{Probability Distribution of the Random Tuple}} Here, we will show two important claims. Claim \ref{claim8.2} that the probability that a fixed $(\oz_1,\dots, \oz_t)$ is chosen by the random process is roughly the same. Claim \ref{claimNeighbors} is a technical claim necessary for the proof of Claims \ref{claim: martingaleexp} and \ref{claim: martingaleconc}.
\\
Throughout this subsection, $(z_1,\dots, z_t)$ always denotes the \textit{random} tuple generated by our random process while $(\oz_1,\dots, \oz_t)$ denotes a fixed tuple in $\overrightarrow{\cB}_\gamma(e;\bj)$. \\
For  every $i$ such that $3\leq i\leq t$, let
\begin{align*}
    p_i := \prod_{s=2}^{i-1} \frac{1}{\mu_s^{j_s}}.
\end{align*}
\begin{claim}
\label{claim8.2}
    For every $i$, such that $3\leq i\leq t$, and $(\oz_1,\dots, \oz_i)\in \overrightarrow{\cB}^{(i)}_\gamma(e;\bj)$, the random tuple $(z_1,\dots, z_i)$ generated by our random process, satisfies: 
    \begin{align*}
        \PP((z_1,\dots, z_i) = (\oz_1,\dots, \oz_i)) = (1\pm n^{-1/6})p_i.
    \end{align*}
\end{claim}

\begin{proof}
    Fix $i$ such that $3\leq i\leq t$ and fix $(\oz_1,\dots, \oz_i)\in \overrightarrow{\cB}^{(i)}_\gamma(e;\bj)$. Since $G$ satisfies $\sP(\alpha,t)$ (see Definition \ref{def: PropertyP}),  and in view of (\ref{ref: randomprocess}), for any $s$ such that $2\leq s\leq i-1$, 
    \begin{align*}
        \PP(z_{s+1} = \oz_{s+1}| z_1 = \oz_1, \dots, z_s = \oz_s) = \frac{1}{|W_{\{\oz_1,\dots, \oz_s\}}^{j_s}|}= \frac{1}{(1\pm n^{-1/5})\mu_s^{j_s}}. 
    \end{align*}
    Thus we have, 
    \begin{align*}
        \PP((z_1,\dots, z_i) = (\oz_1,\dots, \oz_i)) &= \prod_{s=2}^{i-1} \PP(z_{s+1} = \oz_{s+1}| z_1 = \oz_1,\dots, z_s = \oz_s)\\
        &= \prod_{s=2}^{i-1} \frac{(1\pm n^{-1/5.5})}{\mu_i^{j_i}}= (1\pm n^{-1/5.5})^ip_i\\
        &= (1\pm n^{-1/6})p_i.
    \end{align*}
\end{proof}
\begin{corollary} Let $(z_1,\dots, z_t)$ be the random tuple generated by our random process. Then, 
\label{cor: martingale2}
\begin{align*}
    \PP((z_1,\dots, z_t) \in \overrightarrow{\cQ}_\gamma(e;\bj)) = (1\pm n^{-1/7})\frac{|\overrightarrow{\cQ}_\gamma(e;\bj)|}{|\overrightarrow{\cB}_\gamma(e;\bj)|}.
\end{align*}
\end{corollary}
\begin{proof}
    We know that, 
    By Claim \ref{claim8.2}, this implies
    \begin{align*}
      1 = \sum_{(\oz_1,\dots, \oz_t)\in \overrightarrow{\cB}_\gamma(e;\bj)}\PP((z_1,\dots, z_t) = (\oz_1,\dots, \oz_t)) = |\overrightarrow{\cB}_\gamma(e;\bj)|(1\pm n^{-1/6})p_t. 
    \end{align*}
    Consequently, 
    \begin{align*}
        \PP((z_1,\dots, z_t)\in \overrightarrow{\cQ}_\gamma(e;\bj)) &=  \sum_{(\oz_1,\dots, \oz_t)\in \overrightarrow{\cQ}_\gamma(e;\bj)}\PP((z_1,\dots, z_t) = (\oz_1,\dots, \oz_t))\\&=|\overrightarrow{\cQ}_\gamma(e;\bj)|(1\pm n^{-1/6})p_t = (1\pm n^{-1/7})\frac{|\overrightarrow{\cQ}_\gamma(e;\bj)|}{|\overrightarrow{\cB}_\gamma(e;\bj)|}.
    \end{align*}
\end{proof}
Finally, we will also prove the following technical claim, which will be necessary for the calculations that follow. 
\begin{claim}
\label{claimNeighbors}
    Let $s,r,i$ be integers such that $1\leq s\leq r\leq i \leq t$. Then for any tuple $(\oz_1,\dots, \oz_{r})\in \overrightarrow{\cB}^{(r)}_\gamma(e;\bj)$, we have:
    \begin{align*}
        \PP(z_{i+1}\in N_G(\oz_{s})|\oz_1,\dots, \oz_{r}) = (1\pm n^{-1/7})\frac{j_{i}}{i}. 
    \end{align*}
\end{claim}
\begin{proof} Let $s,i,r$ and $(\oz_1,\dots, \oz_{r})$ be as mentioned in the statement. By the law of total probability, we have:
    \begin{align}
         &\PP(z_{i+1}\in N_G(\oz_{s})|\oz_1,\dots, \oz_{r})\nonumber\\ 
         \label{eqn: 8.4eqn}
         &= \sum_{(\oz_{r+1},\dots, \oz_i)} \PP(z_{i+1}\in N_G(\oz_{s})|\oz_1, \dots, \oz_i)\PP(z_{r+1}= \oz_{r+1}, \dots, z_i = \oz_i|\oz_1,\dots, \oz_{r}),
    \end{align}
where the sum is over all possible choices of $(\oz_{r+1},\dots, \oz_i)$ by the random process. For a fixed $(\oz_{r+1},\dots, \oz_i)$ in the sum, let $S = \{\oz_1,\dots, \oz_i\}$. In view of (\ref{ref: randomprocess}) and the fact that $G$ satisfies $\sP(\alpha,t)$ (Definition \ref{def: PropertyP}), we have: 

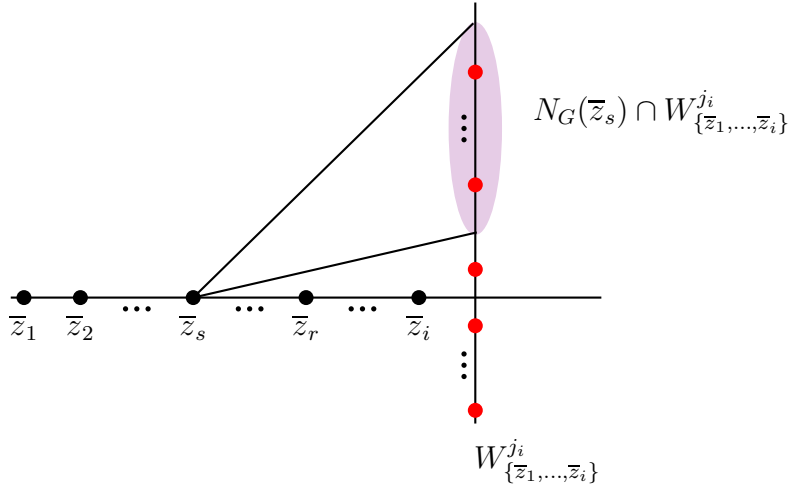
\begin{figure}[t!]
\begin{tikzpicture}[thick,
  fsnode/.style={circle, fill=black, inner sep = 2pt},
  rsnode/.style={circle, fill= red, inner sep = 2pt},
  gsnode/.style={circle, fill = black, inner sep = 0.7pt},
  ssnode/.style={circle, fill=red, inner sep = 2.5pt},
    shorten >= -5pt,shorten <= -5pt, scale = .75
]
\tikzset{hfit/.style={rounded rectangle, inner xsep=0pt, fill=#1!30},
           vfit/.style={rounded corners, fill=#1!30}}
\draw[thick] (3,5) -- (3,-2);
\draw[thick] (-5,0) -- (5,0);

\node[rsnode] at (3,4) (a) {};
\node[rsnode] at (3,2) (b) {};

\node[gsnode] at (2.8,3.2) {};
\node[gsnode] at (2.8,3) {};
\node[gsnode] at (2.8,2.8) {};

\node[rsnode] at (3,0.5) (a1) {};
\node[rsnode] at (3,-0.5) (b1) {};

\node[gsnode] at (2.8,-1.2) {};
\node[gsnode] at (2.8,-1) {};
\node[gsnode] at (2.8,-1.4) {};

\node[rsnode, label={[xshift=2em, yshift= -3em] $W_{\{\oz_1,\dots, \oz_i\}}^{j_i}$}] at (3,-2) (b2) {};

\node[fsnode, label={[xshift=0em, yshift= -2em] $\oz_1$}] at (-5,0) (d) {};
\node[fsnode, label={[xshift=0em, yshift= -2em] $\oz_2$}] at (-4,0) (e) {};

\node[gsnode] at (-3.0,-0.2) {};
\node[gsnode] at (-3.2,-0.2) {};
\node[gsnode] at (-2.8,-0.2) {};

\node[fsnode, label={[xshift=0em, yshift= -2em] $\oz_s$}] at (-2,0) (c) {};

\node[gsnode] at (1.0,-0.2) {};
\node[gsnode] at (1.2,-0.2) {};
\node[gsnode] at (0.8,-0.2) {};

\node[fsnode, label={[xshift=0em, yshift= -2em] $\oz_r$}] at (0,0) (f) {};

\node[gsnode] at (-1.0,-0.2) {};
\node[gsnode] at (-1.2,-0.2) {};
\node[gsnode] at (-0.8,-0.2) {};

\node[fsnode, label={[xshift=0em, yshift= -2em] $\oz_i$}] at (2,0) (g) {};

\node at (3,4.9) (top) {};
\node at (3,1.15) (bottom) {};

\draw[-] (c) -- (top) ;
\draw[-] (c) -- (bottom) ;

\begin{scope}[on background layer]
\node[fit=(a) (b), fill = violet!20, ellipse, label={[xshift=6em, yshift= -4em] $N_G(\oz_{s})\cap W_{\{\oz_1,\dots, \oz_i\}}^{j_i}$}] {};
\end{scope}
\end{tikzpicture}
\caption{Choices for $z_{i+1}$ in the random process}
\end{figure}

\begin{align*}
    &\PP(z_{i+1}\in N_G(\oz_{s})|\oz_1,\dots,  \oz_i)\\
    &= \frac{|W_{S,\oz_{s}}^{j_i}|}{|W_S^{j_i}|} 
    \\&= (1\pm n^{-1/7}) \frac{\omega_{i}^{j_i}}{\mu_i^{j_i}}
    \\&=  (1\pm n^{-1/7})\frac{(n-i){i-1\choose j_i - 1}2^{-i}}{(n-i){i\choose j_i}2^{-i}}\\
    &= (1\pm n^{-1/7})\frac{j_i}{i}. 
\end{align*}
Consequently, (\ref{eqn: 8.4eqn}) is equivalent to, 
\begin{align*}
    &\PP(z_{i+1}\in N_G(\oz_{s})|\oz_1,\dots, \oz_{r})\\&= (1\pm n^{-1/7})\frac{j_i}{i}\sum_{(\oz_{r+1},\dots, \oz_i)}\PP(z_{r+1}=\oz_{r+1}, \dots, z_i=\oz_i|\oz_1,\dots, \oz_{r})\\
    &= (1\pm n^{-1/7})\frac{j_i}{i}.
\end{align*}
\end{proof}
\subsubsection{\textbf{Probability of Random Tuple Being Quasiclique}} Fix $\bj\in \cJ_\gamma(e)$. For any tuple $(\oz_1,\dots, \oz_t)$ in $\overrightarrow{\cB}_\gamma(e;\bj)$, the number of edges induced on $\{\oz_1,\dots, \oz_t\}$ is $T= (\frac{1}{2}+\alpha){t\choose 2}$. To ensure that this tuple is also in $\overrightarrow{\cQ}_\gamma(e;\bj)$, it must be an $\alpha$-quasiclique and thus must satisfy that,
 \begin{align}
 \label{conditionquasiclique}
        |N_G(\oz_s)\cap \{\oz_1,\dots, \oz_t\}| = \left(\frac{1}{2}+\alpha \pm \frac{3\alpha}{4}\right)t, 
\end{align}
for every $s\in [t]$, i.e., the degree of each vertex in the induced subgraph must be $(\frac{1}{2}+\alpha\pm \frac{\alpha}{2})t$.\\
Given a random tuple $(z_1,\dots, z_t)$ generated by the random process, for every $s$ such that $1\leq s\leq t$, let the random variable, 
    \begin{align*}
        d_s = d_s(z_1,\dots, z_t) := |N_G(z_{s})\cap \{z_1,\dots, z_t\}|.
    \end{align*}
We will prove the following statements for $(z_1,\dots, z_t)$ generated by the random process.  
\begin{claim}
\label{claim: martingaleexp}
For every $s$ such that $1\leq s\leq t$,
    \begin{align}
    \label{eqn: expds3}
        \EE[d_s] = \left(\frac{1}{2}+ \alpha \pm (2\gamma+\alpha^2)\right)t.
    \end{align}
\end{claim}
\begin{claim}
\label{claim: martingaleconc}
For every $s$ such that $1\leq s\leq t$,
     \begin{align*}
         \PP(|d_s - \EE[d_s]| >\frac{\gamma}{2} t) < \exp\left(\frac{-\gamma^2}{32}t\right).
    \end{align*}
\end{claim} 
\begin{corollary} The random tuple $(z_1,\dots, z_t)$ chosen by the random process satisfies, 
\label{cor: martingale}
    \begin{align*}
        \PP((z_1,\dots, z_t)\in \overrightarrow{\cQ}_\gamma(e;\bj)) \geq 1-n^{-2A}.
    \end{align*}
\end{corollary}
First we prove Corollary \ref{cor: martingale} using Claims \ref{claim: martingaleexp} and \ref{claim: martingaleconc}, and then we prove the claims. 
\begin{proof}[Proof of Corollary \ref{cor: martingale}]
    In view of (\ref{conditionquasiclique}), it is sufficient to show that the probability that there exists an $s\in [t]$ such that, 
    \begin{align}
    \label{conditionquasiclique2}
        \left|d_s-\left(\frac{1}{2}+ \alpha \right)t\right| > \frac{3\alpha}{4}t,
    \end{align}
    is at most $n^{-2A}$. Since, $\alpha \leq 1/2$ and $\gamma = \alpha/10$, we have,
    \begin{align*}
        \left(\frac{\gamma}{2} + 2\gamma + \alpha^2\right)t \leq \frac{3\alpha}{4}t.
    \end{align*}\\
    Consequently, in view of Claim \ref{claim: martingaleexp} and Claim \ref{claim: martingaleconc}, the probability that (\ref{conditionquasiclique2}) holds for some $1\leq s\leq t$ is at most, 
    \begin{align*}
        \sum_{s=1}^{t} \PP(|d_s - \EE[d_s]| >\frac{\gamma t}{2}) < t\cdot 2\exp\left(-\frac{\gamma^2}{32} t\right).
    \end{align*}
    In view of (\ref{eqn: 2.6condition2}), we have,
    \begin{align*}
       t\cdot 2\left(-\frac{\gamma^2}{32} t\right) < \exp (-\frac{\gamma^2}{33}t) = \exp \left(-\frac{\alpha^2}{3300}c\alpha^{-2}\log n\right) = n^{-\frac{c}{3300}}.
    \end{align*}
    We have that $A = 10^{-4}c$, and hence $2A \leq c/3300$. Consequently, 
    \begin{align*}
         \PP((z_1,\dots, z_t)\in \overrightarrow{\cQ}_\gamma(e;\bj)) \geq 1-n^{-\frac{c}{3300}} \geq 1-n^{-2A}. 
    \end{align*}
\end{proof}

\begin{proof}[Proof of Claim \ref{claim: martingaleexp}]
Fix $1\leq s\leq t$. In this proof, it will be convenient to extend the backward degree sequence $(j_2,\dots, j_{t-1})$ by $j_0=0$ and $j_1 = 1$.\\
Let $Y_{i}$ be the indicator variable for the event that $z_{i}\in N_G(z_{s})$. We have that, 
    \begin{align}
    \label{eqn: dsexpression}
        d_s = Y_1 + \cdots + Y_{t} = j_{s-1} + Y_{s+1} + \cdots + Y_{t}. 
    \end{align}
    Consequently, 
    \begin{align}
    \label{expds}
        \EE[d_s] = j_{s-1} + \sum_{i=s}^{t-1} \PP(z_{i+1}\in N_G(z_{s})). 
    \end{align}
    Let $i\geq s$. Using Claim \ref{claimNeighbors} with $s\leq r= i$, we have:
    \begin{align*}
       \PP(z_{i+1}\in N_G(z_{s})) &= \sum_{(\oz_1,\dots, \oz_i)}\PP(z_{i+1}\in N_G(\oz_{s})|z_1 = \oz_1, \dots, z_i= \oz_i)\PP(z_1 = \oz_1,\dots, z_i = \oz_i)
       \\&=  (1\pm n^{-1/7})\frac{j_i}{i}\sum_{(\oz_1,\dots, \oz_i)}\PP(z_1 = \oz_1,\dots, z_i = \oz_i)
       \\&=  (1\pm n^{-1/7})\frac{j_i}{i}.
    \end{align*}
    Together, with (\ref{expds}),   and in view of (\ref{eqn: 2.6condition2}), $n^{-1/7}<\gamma$, this implies,
    \begin{align}
    \label{expds2}
        \EE[d_s] = j_{s-1} + \sum_{i=s}^{t-1}  (1\pm \gamma)\frac{j_i}{i}, 
    \end{align}
    Recall that since $\bj \in \cJ_\gamma(e)$, we have that,
    \begin{enumerate}
        \item $0\leq j_i\leq i$ for all $0\leq i\leq t-1.$
        \item $j_i = \left(\frac{1}{2}+\alpha\pm \gamma\right)i$ for all $c\log n \leq i\leq t-1$. 
    \end{enumerate}
    We consider the case where $s-1\leq c\log n$ and $s-1>c\log n$ separately and show the lower bound on $\EE[d_s]$ to complete the proof of Claim \ref{claim: martingaleexp}. The upper bound follows similarly. If $s-1\geq c\log n$, in view of (\ref{expds2}), 
    \begin{align*}
        \EE[d_s] &\geq  \left(\frac{1}{2}+\alpha - \gamma\right)(s-1) + \sum_{i=s}^{t-1}  (1- \gamma)\left(\frac{1}{2}+\alpha - \gamma\right)\\
        &> \left(\frac{1}{2}+\alpha - 2\gamma\right)(t-1) = \left(\frac{1}{2}+\alpha - 2\gamma\right)\left(1-\frac{1}{t}\right)t.
    \end{align*}
    In view of (\ref{eqn: 2.6condition2}), and that $\gamma = \alpha/10$, we have that $\frac{1}{t} < \alpha^2$ and consequently, 
    \begin{align*}
         \EE[d_s] &\geq  \left(\frac{1}{2}+\alpha - (2\gamma + \alpha^2)\right)t.
    \end{align*}
   On the other hand, if $s-1 <  c\log n$, then the contributions of $j_i$ and the terms $j_i/i$ for $i < c\log n$  are negligible. In view of (\ref{expds2}), we have 
    \begin{align*}
         \EE[d_s] &\geq \sum_{i=\lceil c\log n \rceil}^{t-1}  (1- \gamma)\frac{j_i}{i}  \geq (1-\gamma)\left(\frac{1}{2}+\alpha - \gamma\right)(t-c\log n)\\
         &\geq \left(\frac{1}{2}+\alpha - 2\gamma\right)t\left(1-\frac{c\log n}{t}\right)\\
         &\geq  \left(\frac{1}{2}+\alpha - (2\gamma + \alpha^2)\right)t.
    \end{align*}
\end{proof}
\begin{proof} [Proof of Claim \ref{claim: martingaleconc}]
Fix $1\leq s\leq t$. We will show that $d_s$ concentrates around its expected value by considering a sequence of random variables that form a Martingale. Let $d_{s,s}= \EE[d_s]$. For $l$ such that $s+1\leq l\leq t$, let,
     \begin{align*}
         d_{s,l}:= \EE[d_s|z_1,\dots, z_{l}]
     \end{align*}
    i.e., $d_{s,l}$ is the conditional expectation of $d_s$ after exposing vertices $z_1,\dots, z_{l}$. In particular, $d_{s,t} = d_s$. Such a sequence of random variables form a Doob Martingale. We claim that, for every $s\leq l\leq t$,
      \begin{align*}
          |d_{s,l+1}-d_{s,l}| \leq 2. 
      \end{align*}
    Formally, the case $l=s$ is slightly different from the case $l\geq s+1$. First, we consider the case where $l\geq s+1$. Fix a tuple $(\oz_1,\dots, \oz_{l+1})$. In view of (\ref{eqn: dsexpression}), 
    \begin{align*}
        d_{s,l+1}-d_{s,l} &= \EE[d_s|\oz_1,\dots, \oz_{l+1}] - \EE[d_s|\oz_1,\dots, \oz_{l}]\\
        &= \sum_{i=s+1}^{t} (\EE[Y_i|\oz_1,\dots, \oz_{l+1}] - \EE[Y_i|\oz_1,\dots, \oz_{l}]).
    \end{align*}
    Note that for $s+1\leq i\leq l\leq t-1$, 
    \begin{align*}
        \EE[Y_i|\oz_1,\dots, \oz_{l+1}] = \EE[Y_i|\oz_1,\dots, \oz_{l}] = Y_i,
    \end{align*}
    as revealing $z_{i+1}$ determines $Y_i$. Further, $Y_{l+1}$ is determined by $\oz_{l+1}$, and consequently, 
    \begin{align*}
         |\EE[Y_{l+1}|\oz_1,\dots, \oz_{l+1}] - \EE[Y_{l+1}|\oz_1,\dots, \oz_{l}]| \leq 1.
    \end{align*}
    Consequently, 
    \begin{align}
         |d_{s,l+1}-d_{s,l}| &\leq 1+ \sum_{i=l+2}^{t} |\EE[Y_i|\oz_1,\dots, \oz_{l+1}] - \EE[Y_i|\oz_1,\dots, \oz_{l}]|\nonumber\\
         \label{eqn: sumLipschitz}
         &= 1 +\sum_{i=l+1}^{t-1} |\PP(z_{i+1}\in N_G(\oz_{s})|\oz_1,\dots, \oz_{l+1}) - \PP(z_{i+1}\in N_G(\oz_{s})|\oz_1,\dots, \oz_{l})|.
    \end{align}
    If $i$ is such that $l+1\leq i\leq t-1$, then applying Claim \ref{claimNeighbors} with $r = l$ and $r=l+1$, respectively, we have that, 
    \begin{align}
    \label{eqn: lipschitz2}
         |\PP(z_{i+1}\in N_G(\oz_{s})|\oz_1,\dots, \oz_{l+1}) - \PP(z_{i+1}\in N_G(\oz_{s})|\oz_1,\dots, \oz_{l})| < 2n^{-1/7}.
    \end{align}
Consequently, in view of (\ref{eqn: 2.6condition2}) and (\ref{eqn: sumLipschitz}),
\begin{align*}
    |d_{s,l+1}-d_{s,l}| < 1 + t\cdot 2n^{-1/7} < 2.
\end{align*}
For the case where $l = s$, a similar argument as above implies, 
\begin{align*}
   |d_{s,s}-d_{s,s+1}| &= |\EE[d_s]- \EE[d_s|z_1,\dots, z_{s+1}]|\\
   &\leq  |\EE[d_s]- \EE[d_s|z_1,\dots, z_{s}]|+|\EE[d_s|z_1,\dots, z_s]-\EE[d_s|z_1,\dots, z_{s+1}]| \\& \leq t\cdot n^{-1/7} + 1 < 2.  
\end{align*}
Applying Azuma's Inequality (Lemma \ref{thm: Azuma}) with $\lambda = \gamma t/2$ to the sequence of $(t-s)$ variables, $\EE[d_s] = d_{s,s},\dots,d_{s,t} = d_s$, we get, 
    \begin{align*}
      \PP\left(|d_s - \EE[d_s]| >\frac{\gamma t}{2}\right) &< 2\exp\left(-\frac{\gamma^2 t^2}{2\cdot 4\cdot 4(t-s)}\right)\\
        &< 2\exp\left(-\frac{\gamma^2}{32} t\right).
    \end{align*}
\end{proof}
\begin{proof}[Proof of Proposition \ref{thm: martingalelemma}]
    Together, Corollary \ref{cor: martingale} and Corollary \ref{cor: martingale2} imply, 
    \begin{align*}
     (1+n^{-1/7})\frac{|\overrightarrow{\cQ}_\gamma(e;\bj)|}{|\overrightarrow{\cB}_\gamma(e;\bj)|}  \geq   \PP((z_1,\dots, z_t) \in \overrightarrow{\cQ}_\gamma(e;\bj)) \geq 1-n^{-2A},
    \end{align*}
    which implies (\ref{eqn: MartingaleMain}), i.e., 
    \begin{align*}
    |\overrightarrow{\cQ}_\gamma(e;\bj)| = (1\pm n^{-A})|\overrightarrow{\cB}_\gamma(e;\bj)|.
    \end{align*}
    Summing over all $\bj\in \cJ_\gamma(e)$, implies Proposition \ref{thm: martingalelemma}.
\end{proof}

\subsection{Proof of Proposition \ref{thm: multiplicity}} 
Fix $\gamma = \alpha$. We need to show that, 
\begin{align*}
    |\overrightarrow{\cQ}_\alpha(e)|= (1\pm n^{-3A\alpha^2})(t-2)!|\cQ(e)|.
\end{align*}
For a $Q=\{u_1, u_2, \dots, u_t\}$ in $\cQ(e)$, and consider the set, $Q^t \cap \overrightarrow{\cB}_\alpha(e)$ i.e., the $t$-tuples formed by vertices of $Q$ that are in $\overrightarrow{\cB}_\alpha(e)$. Since $u_1,u_2$ must be the first two elements in a tuple in $\overrightarrow{\cB}_\alpha(e)$, we must have 
\begin{align*}
    |Q^t \cap \overrightarrow{\cB}_\alpha(e)|\leq (t-2)!.
\end{align*}
Further, recall that $\overrightarrow{\cQ}_\alpha(e)$ is the collection of all tuples in $|\overrightarrow{\cB}_\alpha(e)|$ whose underlying vertex sets are in $\cQ(e)$, and consequently, 
\begin{align}
\label{eqn: multiplicitymain}
    |\overrightarrow{\cQ}_\alpha(e)|= \left|\bigcup_{Q\in \cQ(e)}Q^t \cap \overrightarrow{\cB}_\alpha(e) \right|=\sum_{Q\in\cQ(e)} |Q^t \cap \overrightarrow{\cB}_\alpha(e)|.
\end{align}
We will show that, 
\begin{claim}
\label{claim: multiplicity}
    For every $Q\in \cQ(e)$, 
    \begin{align*}
        |Q^t \cap \overrightarrow{\cB}_\alpha(e)| \geq (1- n^{-3A\alpha^2})(t-2)!.
    \end{align*}
\end{claim}
Equation (\ref{eqn: multiplicitymain}) and Claim \ref{claim: multiplicity} together imply Proposition \ref{thm: multiplicity}, i.e., 
\begin{align*}
     |\overrightarrow{\cQ}_\alpha(e)|= (1\pm n^{-3A\alpha^2})(t-2)!|\cQ(e)|.
\end{align*}
Before we prove the above claim, we outline the proof. We consider the $t$-tuples formed by all $(t-2)!$ permutations of the vertices of $Q=\{u_1,\dots, u_t\}$ that fix the edge $e = \{u_1,u_2\}$,and show that at least $(1- n^{-3A\alpha^2})(t-2)!$ of them are in $\overrightarrow{\cB}_\alpha(e)$, i.e, it has a \textit{backward degree sequence} that is in $\cJ_\alpha(e)$. 
\begin{proof}
Fix $Q=\{u_1, u_2, \dots, u_t\}$ in $\cQ(e)$. Let $(x_1,\dots, x_t)\in Q^t$ such that $x_1= u_1$ and $x_2 = u_2$. The tuple $(x_1,\dots, x_t)\in \overrightarrow{\cB}_\alpha(e)$, if and only if it has its backward degree sequence in $\cJ(e)$, i.e.,
\begin{align}
\label{eqn: backwarddegcondition}
    (|N_G(x_{s+1})\cap \{x_1,\dots, x_s\}|)_{s=2}^{t-1} \in \cJ_\alpha(e). 
\end{align}
Let $\sigma$ be a permutation chosen uniformly at random from the set of permutations of $\{1,\dots, t\}$ that fixes $\{1,2\}$. 
Let $(x_1,\dots, x_t)$ be the random tuple with $x_i = u_{\sigma(i)}$ for $1\leq i\leq t$. Fix $s$ such that $2\leq s\leq t-1$ and define the random variable, 
\begin{align*}
    j_s = j_s(x_1,\dots, x_t) := |N_G(x_{s+1})\cap \{x_1,\dots, x_s\}|.
\end{align*}
In view of (\ref{eqn: backwarddegcondition}) and Definition \ref{def: gammasequences}, the probability that $(x_1,\dots, x_t)$ is in $\overrightarrow{\cB}_\alpha(e)$ is equal to the probability that for every $c\log n\leq s \leq t-1$,
\begin{align}
\label{eqn: Zsrange}
    j_s &= \left(\frac{1}{2}+\alpha \pm \alpha \right)s.
\end{align}
On the other hand, each permutation of vertices of $Q$ with $x_1 = u_1$ and $x_2 = u_2$ is chosen probability $1/(t-2)!$ and hence, 
\begin{align*}
    \PP((x_1,\dots, x_t)\in Q^t\cap \overrightarrow{\cB}_\alpha(e))= \frac{|Q^t\cap \overrightarrow{\cB}_\alpha(e)|}{(t-2)!}.
\end{align*}
We will now show that, 
\begin{align*}
    \PP((x_1,\dots, x_t)\in Q^t\cap \overrightarrow{\cB}_\alpha(e))\geq 1- n^{-3A\alpha^2},
\end{align*}
which is equivalent to showing that the probability that (\ref{eqn: Zsrange}) holds for every $s$ such that $c\log n \leq s\leq t-1$ is at least $1- n^{-3A\alpha^2}$. Fix $2\leq i\leq t$. Let us condition on the event $x_{s+1}=u_i$. We will show the following statements, 
\begin{enumerate}
    \item We have,
    \begin{align}
    \label{eqn: j_sconditional}
    \EE[j_s|x_{s+1} = u_i] = \left(\frac{1}{2}+ \alpha \pm \frac{4\alpha}{5}\right)s.
    \end{align}
    \item Further, 
    \begin{align}
    \label{eqn: j_sconcconditional}
        \PP\left(|j_s - \EE[j_s|x_{s+1}= u_i]|> \frac{\alpha s}{5}\mid x_{s+1}=u_i\right)<2\exp\left(-\frac{\alpha^2 s}{75}\right).
    \end{align}
\end{enumerate}
First we show (1). Conditioned on the event $x_{s+1}=u_i$, the random variable $j_s$ is the size of $\{x_1,\dots, x_s\}\cap N_G(u_i)$, where $x_1,\dots, x_s$ are chosen from $Q\setminus \{u_i\}$. Let, 
\begin{align*}
    d'(u_i) = |N(u_i)\setminus \{u_1,u_2\}|. 
\end{align*}
Since $x_1 = u_1, x_2 = u_2$ are fixed, $j_s$ conditioned on the event $x_{s+1}=u_i$, follows the hypergeometric distribution $H(t-3,s-2, d'(u_i))$. 
\\
Since $Q\in \cQ(e)\subseteq \cQ$, in view of Definition \ref{def: quasicliques}, we have: 
\begin{align*}
    d'(u_i) &= d(u_i)\pm 2 = \left(\frac{1}{2}+\alpha \pm \frac{3\alpha}{4}\right)t\pm 2\\
    &=\left(\frac{1}{2}+\alpha \pm \frac{31\alpha}{40}\right)t . 
\end{align*}
Consequently, in view of the fact that $s \geq c\log n$ and $t = c\alpha^{-2}\log n$, we have (\ref{eqn: j_sconditional}), i.e.,  
\begin{align}
    \EE[j_s|x_{s+1} = u_i] &=
    \frac{d'(u_i)(s-2)}{(t-3)}\nonumber\\
    &= \left(\frac{1}{2}+ \alpha \pm \frac{31\alpha}{40}\right)t\frac{(s-2)}{(t-3)}\nonumber\\
    \label{eqn: expmultlessthan}
    &= \left(\frac{1}{2}+ \alpha \pm \frac{4\alpha}{5}\right)s < s.
\end{align}
Since the random variable $j_s$ conditioned on the event $x_{s+1} = u_i$ follows the hypergeometric distribution, we can use Lemma \ref{thm: Chernoff1} with $\epsilon = \alpha s/5\EE[j_s|x_{s+1}=u_i]$. In view of (\ref{eqn: expmultlessthan}), this implies (\ref{eqn: j_sconcconditional}), i.e.,   
\begin{align*}
    \PP\left(|j_s - \EE[j_s|x_{s+1}= u_i]|> \frac{\alpha s}{5}\mid x_{s+1}=u_i\right) &< 2\exp\left(-\frac{\alpha^2 s^2}{75\EE[j_s|x_{s+1}=u_i]}\right)\\
    &<  2\exp\left(-\frac{\alpha^2 s}{75}\right).
\end{align*}
Consequently, in view of (\ref{eqn: j_sconditional}) and (\ref{eqn: j_sconcconditional}),  
\begin{align*}
    &\PP\left(\left|j_s - \left(\frac{1}{2}+\alpha\right)s\right| \geq \alpha s\right)\\ 
    &= \sum_{i=3}^t\PP\left(\left|j_s - \left(\frac{1}{2}+\alpha\right)s\right| \geq\alpha s|x_{s+1}=u_i\right)  \PP(x_{s+1} = u_i)  \\
    &< 2\exp\left(-\frac{\alpha^2 s}{75}\right)\sum_{i=3}^t \PP(x_{s+1} = u_i)= 2\exp\left(-\frac{\alpha^2 s}{75}\right).
\end{align*}
Consequently, the probability that there exists $s$ such that $c\log n\leq s\leq t-1$ and,
\begin{align*}
    \left|j_s - \left(\frac{1}{2}+\alpha\right)s\right| \geq \alpha s,
\end{align*}
is at most,
\begin{align*}
   \sum_{s\geq c\log n}^{t-1}\PP\left(\left|j_s - \left(\frac{1}{2}+\alpha\right)s\right| \geq \alpha s\right) < t\cdot 2\exp\left(-\frac{\alpha^2 c\log n}{75}\right).
\end{align*}
In view of (\ref{eqn: 2.6condition2}) and that $A = 10^{-4}c$, we have that, 
\begin{align*}
    t\cdot 2\exp\left(-\frac{\alpha^2 c\log n}{75}\right) < 2t n^{-\frac{\alpha^{2}c}{75}}< n^{-\frac{\alpha^{2}c}{100}} < n^{-3A\alpha^{2}}. 
\end{align*}
Thus the probability that $(x_1,\dots, x_t)$ is in $\overrightarrow{\cB}_\alpha(e)$ is at least $1-n^{-3A\alpha^{2}}$. And consequently, 
\begin{align*}
   |Q^t\cap \overrightarrow{\cB}_\alpha(e)| \geq (1- n^{-3A\alpha^{2}})(t-2)!.
\end{align*}
\end{proof}


\bibliography{literature}

\end{document}